\documentclass[a4paper,11pt]{article}
\usepackage[margin=1in]{geometry}
\usepackage[utf8]{inputenc}
\usepackage[english]{babel}
\usepackage[T1]{fontenc}
\usepackage{color}    
\usepackage{graphicx} 
\usepackage{listings}
\usepackage{tikz}
\usetikzlibrary{positioning}
\usepackage{fancyhdr} 
\usepackage{lastpage} 
\usepackage{amsmath,amssymb,mleftright} 
\usepackage{amsthm}
\usepackage{authblk}
\usepackage{listings} 
\usepackage{parcolumns} 
\usepackage{array,tabu,booktabs} 
\usepackage{multirow} 
\usepackage{pgfplots,wrapfig}
\usepackage{titling}  
\usepackage{lmodern}  
\usepackage{verbatim,verbatimbox} 
\usepackage{fancyvrb} 
\usepackage{enumitem} 
\PassOptionsToPackage{hyphens}{url}\usepackage{hyperref} 
\usepackage{float}
\usepackage{caption}
\usepackage{subcaption}
\usepackage{animate}
\usepackage[final]{pdfpages}
\usepackage{diagbox}
\usepackage{hhline}
\usepackage{titlesec}
\usepackage{algpseudocode}
\usepackage[export]{adjustbox}
\usepackage{bigints}
\usepackage[makeroom]{cancel}
\usepackage{xparse}
\usepackage{setspace}
\usepackage{mathtools}
\titlespacing*{\section}
{0pt}{5.5ex plus 1ex minus .2ex}{4.3ex plus .2ex}
\titlespacing*{\subsection}
{0pt}{5.5ex plus 1ex minus .2ex}{4.3ex plus .2ex}
\usepackage{rotating}
\usepackage{cleveref}

\usepackage{titlesec}
\usepackage{csquotes}

\titleformat{\paragraph}
{\normalfont\normalsize\bfseries}{\theparagraph}{1em}{}
\titlespacing*{\paragraph}
{0pt}{3.25ex plus 1ex minus .2ex}{1.5ex plus .2ex}

\makeatletter
\newcommand*\bigcdot{\mathpalette\bigcdot@{.5}}
\newcommand*\bigcdot@[2]{\mathbin{\vcenter{\hbox{\scalebox{#2}{$\m@th#1\bullet$}}}}}
\makeatother

\makeatletter
\def\BState{\State\hskip-\ALG@thistlm}
\makeatother
\definecolor{tempcolor}{rgb}{0.7773 0.0820 0.5195}

\newcommand{\lp}{\left(}
\newcommand{\rp}{\right)}
\newcommand{\lb}{\left[}
\newcommand{\rb}{\right]}

\newcommand{\R}{\mathbb{R}}

\pgfplotsset{compat=1.18}

\newcommand{\norm}[1]{\ensuremath{\left\lVert  #1\right\rVert}} 
\newcommand{\abs}[1]{\ensuremath{\left\vert  #1\right\vert}}

\usepackage{xcolor}
\definecolor{dkgreen}{rgb}{0,0.6,0}
\definecolor{dred}{rgb}{0.545,0,0}
\definecolor{dblue}{rgb}{0,0,0.545}
\definecolor{lgrey}{rgb}{0.9,0.9,0.9}
\definecolor{gray}{rgb}{0.4,0.4,0.4}
\definecolor{darkblue}{rgb}{0.0,0.0,0.6}
\definecolor{turquoise}{rgb}{0.2500,0.8750,0.8125}
\definecolor{indigo}{rgb}{0.2930, 0,0.5078}
\definecolor{mag}{rgb}{1, 0,1}
\definecolor{corn}{rgb}{0.3906,0.5820,0.9258}
\definecolor{mvr}{rgb}{0.7773,0.0820,0.5195}
\definecolor{dod}{rgb}{0.11719,0.5625,1}
\lstdefinelanguage{MatlabCostum}{
      backgroundcolor=\color{white},  
      basicstyle=\footnotesize \ttfamily \color{black} \bfseries,   
      breakatwhitespace=false,       
      breaklines=true,               
      captionpos=b,                   
      commentstyle=\color{dkgreen},   
      emph={repmat, ones},			
      keywordstyle=\color{blue},           
      escapeinside={\%*}{*)},                  
      frame=single,                  
      language=Matlab,                
      identifierstyle=\color{black},
      stringstyle=\color{blue},      
      numbers=left,                 
      numbersep=5pt,                  
      numberstyle=\tiny\color{black}, 
      rulecolor=\color{black},        
      showspaces=false,               
      showstringspaces=false,        
      showtabs=false,                
      stepnumber=1,                   
      tabsize=4,
      title=\lstname
}
\newtheorem{theorem}{Theorem}[section]

\newtheorem{lemma}[theorem]{Lemma}

\newtheorem{remark}[theorem]{Remark}

\usepackage[colorinlistoftodos,prependcaption,textsize=tiny]{todonotes}

\allowdisplaybreaks

\usepackage[backend=biber,
            style=numeric,
            abbreviate=false,
            dateabbrev=false,
            alldates=long]{biblatex}

\begin{document}

\title{Bayesian Formulation of Acousto-Electric Tomography and quantified uncertainty in limited view}

\author{Hjørdis Schl\"uter$^1$, Babak Maboudi Afkham$^2$ }

\maketitle
\footnotetext[1]{Department of Mathematics and Statistics, University of Helsinki, Helsinki, Finland.}
\footnotetext[2]{Research Unit of Mathematical Sciences, University of Oulu, Oulu, Finland}

\begin{abstract}
Acousto–electric tomography (AET) is a hybrid imaging modality that combines electrical impedance tomography with focused ultrasound perturbations to obtain interior power density measurements, which provide additional information that can enhance the stability of conductivity reconstruction. In this work, we study the AET inverse problem within a Bayesian framework and compare statistical reconstruction with analytical approaches. The unknown conductivity is modeled as a random field, and inference is based on the posterior distribution conditioned on the measurements. We consider likelihood constructions based on both $L^1$- and $L^2$-type data misfit norms and establish Bayesian well-posedness for both formulations within the framework of Stuart (2010). Numerical experiments investigate performance of the Bayesian method from noisy power density measurements using the $L^1$ and $L^2$ likelihood functions and a smooth prior and a piecewise-constant prior for different limited view configurations, including severely limited boundary access. In particular, we demonstrate that small inclusions near the accessible boundary can be reconstructed from AET data corresponding to a single EIT measurement, and we quantify reconstruction uncertainty through posterior statistics.
\end{abstract}

\section{Introduction}
Electrical impedance tomography (EIT) is a non-invasive imaging modality for reconstructing the interior conductivity distribution of an object (e.g., a human body) from boundary measurements of current and voltage. In a typical EIT setup, electrodes are attached to the boundary of the object. Known electrical currents (or voltages) are applied through these electrodes, inducing an electric potential inside the object that satisfies an elliptic partial differential equation. The resulting boundary voltages (or currents) are then measured. The EIT inverse problem consists of recovering the interior electrical conductivity, which depends on the underlying material properties, such as biological tissue, from these noisy boundary measurements. If access is restricted to only a portion of the boundary, we refer to this as a limited-view EIT problem, which consists of recovering the interior conductivity from current and voltage measurements available only on the accessible part of the boundary.

Despite extensive investigation from mathematical, computational, and applied perspectives, the EIT inverse problem remains severely ill-posed. In particular, for sufficiently smooth conductivities, the best possible stability estimate is of logarithmic type, as shown in \cite[Thm.~1]{Alessandrini1988} (for three or higher dimensional objects) together with the optimality result in \cite[Thm.~1]{Mandache2001} (for two or higher dimensional objects). The EIT inverse problem in limited view is even more ill-posed, since for sufficiently smooth conductivities the stability estimate deteriorates to log–log type (for three or higher dimensional objects). Consequently, even small measurement errors in full-view EIT may result in large reconstruction errors, and this effect can be even more pronounced in the limited-view setting. This indicates that the achievable spatial resolution is intrinsically limited.

In this paper, we investigate the mathematical properties of acousto–electric tomography (AET), a hybrid imaging modality that combines electrical impedance tomography (EIT) measurements with focused ultrasound perturbations to obtain interior power density measurements. In an AET setup, electrodes are placed on the boundary of the object, as in standard EIT. In addition, a focused ultrasound beam is transmitted into the interior, where it induces a localized perturbation of the electrical conductivity. This perturbation modulates the boundary voltage (or current) measurements, thereby encoding information about the interior conductivity distribution \cite{Zhang:Wang:2004, ammari2008a}.

From a mathematical perspective, the ultrasound-induced modulation acts as a localized interior probe, providing access to internal functionals such as the power density. The availability of such interior data fundamentally alters the nature of the inverse problem. In contrast to classical EIT, which exhibits logarithmic stability, the inverse problem of reconstructing the conductivity from power density measurements is only mildly ill-posed. In particular, under suitable assumptions, Lipschitz stability estimates can be established \cite[Thm. 3.2]{Bal2013} (for two-dimensional objects). This improvement in stability translates into enhanced resolution and robustness in conductivity reconstruction compared to standard EIT. In particular, in limited-view settings, where EIT exhibits only log–log stability, AET provides a promising approach for obtaining improved reconstructions.

Typical reconstruction strategies for the AET inverse problem fall into three main categories: analytical reconstruction methods, deterministic variational approaches, and statistical (including Bayesian) methods. All three approaches are grounded in the same forward model, namely the elliptic boundary value problem governing the electrical potential, together with internal power density measurements model. Below we briefly summarize each reconstruction approach.

Analytical reconstruction methods \cite{Monard2012, kuchment2011a, Gebauer2008} exploit structural identities of the coupled PDE system and the internal data. Through a sequence of algebraic and differential manipulations, one derives explicit reconstruction formulas and equations for the conductivity. These methods are closely aligned with the underlying PDE theory and often yield rigorous results on uniqueness and stability. However, they typically rely on idealized assumptions (e.g., smooth coefficients and exact data), and their practical implementation in the presence of noise remains challenging.

In the deterministic variational reconstruction case, the conductivity is reconstructed through a regularized optimization method. This approach allows noise to be incorporated and artifacts can be damped by the regularizer \cite{adesokan2019a, Hoffmann2014, Zhu2025}. However, beyond idealized assumptions mentioned before, it is unclear how much trust we can put in a reconstruction. This issue is amplified when we are outside the conventional AET setup, where access is only possible to parts of the boundary.

In statistical reconstruction, the AET inverse problem \cite{Zhang2017} is formulated within a probabilistic framework in which the unknown conductivity is modeled as a random field. The solution is characterized by the conditional probability distribution of the conductivity given the measurements, commonly referred to as the posterior distribution. This framework enables a systematic incorporation of measurement noise and prior information into the inference procedure. Moreover, it provides credible intervals or credible conductivity profiles, which quantify the uncertainty of the reconstruction in a given experimental setup. In practice, this approach is often used to compute a maximum a posteriori (MAP) estimate of the conductivity. However, a comprehensive interpretation and analysis of the full posterior distribution, particularly in nonstandard or nonconventional AET settings, remains an active area of research.

In this work, we formulate the AET inverse problem in two dimensions within a Bayesian framework. The forward operator is approximated using the finite element method (FEM), and we provide a detailed account of how this discretization is incorporated into the Bayesian formulation. Unlike many Bayesian inverse problems that assume pointwise measurements, AET data are naturally modeled as functions over the spatial domain. We therefore derive translation procedures that relate noisy functional data to pointwise measurements within the finite element discretization, and conversely clarify how the computational discretization must be designed when only pointwise observations are available. Furthermore, we investigate the construction of the likelihood function using both $L^1$- and $L^2$-based data misfit norms. Within the framework of \cite{stuart2010inverse}, we rigorously establish Bayesian well-posedness for both choices. While the $L^2$-based formulation is standard in the literature, the $L^1$-based likelihood has received comparatively less theoretical attention.

We present numerical experiments to assess the performance of the Bayesian method from noisy measurements using the $L^1$ and $L^2$ likelikood functions and a smooth and piecewise-constant prior for varying AET configurations, including scenarios with severely restricted boundary access. After selecting the likelihood function and prior distribution that is most suitable for limited view AET, we compare the Bayesian reconstructions with analytical reconstruction methods \cite{monard2012a} in order to highlight the respective advantages and limitations of both approaches. Our results demonstrate a pronounced discrepancy between the reconstructions obtained under the two noise models considered above. Moreover, we quantify reconstruction uncertainty through the posterior standard deviation, thereby providing a systematic assessment of uncertainty propagation.

Although limited angle settings are considered in \cite{Salo2022,Schlter2025, jensen2023a, Hubmer2018}, these approaches assume that the AET data corresponds to at least two EIT measurements. In contrast, in this work we emphasize that small inclusions located near the accessible part of the boundary in limited-view configurations can already be reconstructed within the Bayesian framework from AET data corresponding to a single EIT measurement. To the best of the authors’ knowledge, this also constitutes the first comprehensive investigation of reconstruction uncertainty in this AET setting.

The main contributions of this work are the finite element discretization of the forward problem (\Cref{sec:FEM}), the Bayesian formulation of the AET inverse problem (\Cref{sec:bayes}), the analysis of Bayesian well-posedness (\Cref{sec:Wellposedness}), and the numerical implementation and investigation of the proposed framework (\Cref{sec:Numerics}).

This paper is organized as follows. In \Cref{sec:MathFormAET}, we review the mathematical formulation of the AET problem and in \Cref{sec:deterministic} we present a analytical reconstruction method. In \Cref{sec:FEM}, we describe the FEM implementation of the forward problem in detail. The Bayesian formulation of the AET inverse problem, including the introduction of the $L^1$- and $L^2$-based noise models, is presented in \Cref{sec:bayes} and we establish wellposedness of this framework in \Cref{sec:Wellposedness}. Numerical experiments investigating the impact of different noise models, limited boundary access, and comparisons with analytical reconstruction methods are reported in \Cref{sec:Numerics}. Finally, concluding remarks are given in \Cref{sec:conclusion}.

%

\section{Mathematical Formulation of AET}
\label{sec:MathFormAET}
The goal of AET is to extract the interior electrical conductivity from energy density measurements based on perturbations caused by ultrasound pressure oscillations imposed from the boundary. Initially, electrodes, connected to the boundary of the object of interest, apply an electric potential. This potential induces a current inside the domain. In the next stage a pulse of ultrasound perturbs electrical conductivity inside the object. In the final stage, sensors measure current modulations, which are proportional to the electrical energy density \cite{ammari2008a,Zhang:Wang:2004}.

To mathematically model this phenomena, let $\Omega\subset \mathbb R^2$, an open, bounded subset with Lipschitz boundary, be our domain of interest. The electrical aspect of AET in limited view can be modeled by an elliptic boundary value problem
\begin{equation}\label{eq:pde-conductivity}
    \begin{cases}
                    \nabla \cdot ( \sigma(\xi) \nabla u(\xi)) = 0, \qquad & \xi \in \Omega,\\
                    u(\xi) = f_i(\xi), &   \xi \in \Gamma_1\subset \partial \Omega,\text{ and } i=1,\dots,d,\\
                    u(\xi) \equiv 0, &  \xi \in \Gamma_2\subset \partial \Omega,
    \end{cases}
\end{equation}
where $\Gamma_1$ and $\Gamma_2$, with $\partial \Omega = \Gamma_1\cup\Gamma_2$, are segments of the boundary of $\Omega$, $\sigma \in C^{\infty} (\Omega)$  is the electrical conductivity with $\sigma_{\text{min}},\sigma_{\text{max}}\in \mathbb R^+$ such that $0<\sigma_{\text{min}}< \sigma < \sigma_{\text{max}}$, $u$ is the electric potential, and $f_i$, $i=1,\dots,d$, are $d$ input voltages at the boundary. If $\Gamma_1 = \partial \Omega$, we obtain the full-view AET setting.

In acousto-electric tomography, the interior quantity of interest is the power density functional
\begin{equation} \label{eq:energy-density}
h_{i,j} = \sigma(\xi) \nabla u_i(\xi) \cdot \nabla u_j(\xi), \qquad i,j =1,\dots,d,
\end{equation}
where $u_i$ and $u_j$ are the electric potentials (solutions to \eqref{eq:pde-conductivity}) corresponding to boundary voltages $f_i$ and $f_j$. This functional represents the local electrical energy density generated by the interaction of the electric fields associated with the boundary inputs $f_i$ and $f_j$. The AET inverse problem consists of two steps: first recovering the interior power densities and then reconstructing the conductivity from them.
In the first step, the quantities $h_{i,j}$ in \eqref{eq:energy-density} are recovered from a finite number $d$ of EIT measurements combined with focused ultrasound waves. The mathematical formulation by which acoustic modulation of the conductivity gives rise to the energy density formulation \eqref{eq:energy-density} is discussed in detail in \Cref{sec:coupled}.

The second step consists of recovering $\sigma(\xi)$ from the measurements $h_{i,j}(\xi)$, for $1 \leq i,j\leq d$. We refer the reader to \cite{monard2012a} for an analytical reconstruction approach that is summarized in section \Cref{sec:deterministic}. To formulate AET as an inverse problem, let $\mathcal G_{i,j}$ denote the mapping $\sigma \mapsto h_{i,j}$, obtained by solving the elliptic problem \eqref{eq:pde-conductivity} and evaluating \eqref{eq:energy-density}. We can now write the forward model for the AET problem as
\begin{equation} \label{eq:forward-exact}
    y_{i,j}(\xi) = \mathcal G_{i,j}[\sigma](\xi) + \varepsilon_{i,j}(\xi), \qquad 1 \leq i,j\leq d,\text{ and } \xi \in \Omega,
\end{equation}
We remark that a physically accurate stochastic noise model for AET 
remains an open research direction. In this work, we assume that 
$\varepsilon_{i,j}$ is a random field satisfying $
\|\varepsilon_{i,j}\|_{L^1(\Omega)} < \infty$ and
$\|\varepsilon_{i,j}\|_{L^2(\Omega)} < \infty$ almost surely. A precise probabilistic specification of $\varepsilon_{i,j}$ 
is given in \Cref{sec:bayes}.

In the following section, we introduce a two-step deterministic reconstruction method for the inverse problem under consideration. This approach requires specific conditions on the boundary input to ensure that a valid and stable reconstruction can be obtained. After presenting this deterministic framework, we return to the problem in a discrete setting by employing a finite element method and reformulating it within a Bayesian framework. This perspective allows us to explicitly account for measurement noise, whose role is typically neglected in the deterministic setting. Finally, in the results section, we provide a direct comparison between the deterministic and Bayesian approaches to highlight their respective performance and limitations.

\section{An Analytical Reconstruction Method} \label{sec:deterministic}

In this section, we present a two-step deterministic reconstruction method for the inverse problem. The method relies on specific conditions imposed on the boundary input to ensure that a valid reconstruction can be achieved. We describe the structure of this approach and outline the assumptions required for its successful application.

We employ the deterministic analytic reconstruction procedure from \cite{monard2012a} to recover the conductivity $\sigma$ from power density measurements $h_{i,j}$ corresponding to $d=2$ voltages imposed such that $1\leq i,j \leq 2$, where $f_1$ and $f_2$ are the imposed boundary inputs. The reconstruction is divided into two steps: first, separating the functionals $\boldsymbol{s}_i=\sqrt{\sigma}\nabla u_i$, and second, recovering $\sigma$ from $\boldsymbol{s}_i=\sqrt{\sigma}\nabla u_i$ for $i=1,2$. These two steps are summarized in the following subsections. Both require inversion of the matrix $[\boldsymbol{H}]_{i,j}:=h_{i,j}$ for $i=1,2$. Invertibility is ensured provided the following Jacobian condition holds:
\begin{equation}\label{eq:JacobiCond}
    \det [\nabla u_1(\xi) \, \nabla u_2(\xi)] > 0, \quad \xi \in \Omega.
\end{equation}
The third subsection lists conditions from the literature that guarantee the existence of boundary functions $f_1$ and $f_2$ such that the corresponding solutions satisfy $\nabla u_i(\xi)\neq 0$ for all $\xi\in \Omega$ and that the Jacobian condition \eqref{eq:JacobiCond} is fulfilled.

\subsection{Reconstruction of \texorpdfstring{$\boldsymbol{s}_i=\sqrt{\sigma} \nabla u_i$}{si=sqrt{sigma}nabla ui}}

Separating the functionals $\boldsymbol{s}_i=\sqrt{\sigma} \nabla u_i$ from the power density measurements $h_{i,j}=\sigma \nabla u_i \cdot \nabla u_j$ for $1 \leq i,j \leq 2$ is achieved by introducing a rotation matrix $\boldsymbol{R}$. The goal is to determine a transformation matrix $\boldsymbol{T}$ (which is known, since it depends only on the data) such that the rotation matrix $\boldsymbol{R}=[\boldsymbol{r}_1 \, \boldsymbol{r}_2]$ can be expressed in terms of $\boldsymbol{S}$ and $\boldsymbol{T}$ as
\[
\boldsymbol{r}_i=\sum_{j=1}^2 t_{i,j}\boldsymbol{s}_j, \quad i=1,2,
\]
or, in matrix notation,
\[
\boldsymbol{R}=\boldsymbol{S}\boldsymbol{T}^T,
\]
where $\boldsymbol{S}=[\boldsymbol{s}_1\ \boldsymbol{s}_2]$. Using structural properties of the rotation matrix and its dependence on the data, one derives equation \eqref{eq:theta} below, which provides a gradient equation for the angle $\theta$ parameterizing $\boldsymbol{R}$. Once $\theta$, and thus $\boldsymbol{R}$, is known, $\boldsymbol{S}$ is reconstructed since $\boldsymbol{T}$ depends solely on the power density data. The matrix $\boldsymbol{T}$ satisfies $\boldsymbol{H}^{-1}=\boldsymbol{T}^T \boldsymbol{T}$, and several choices of $\boldsymbol{T}$ lead to the relation $\boldsymbol{R}=\boldsymbol{S}\boldsymbol{T}^T$. The rotation matrix $\boldsymbol{R}$ is orthogonal with determinant one and can be parameterized by an angle function $\theta$:
\begin{equation*}
    \boldsymbol{R}(\theta)=
    \begin{bmatrix}
        \cos\theta & -\sin\theta\\
        \sin\theta & \cos\theta
    \end{bmatrix}.
\end{equation*}

We introduce the vector fields $\boldsymbol{v}_{i,j}$, which depend on the entries of $\boldsymbol{T}$ and their derivatives:
\begin{equation*}
    \boldsymbol{v}_{i,j}=\nabla(t_{i,1})t^{1,j}+\nabla(t_{i,2})t^{2,j}, \quad 1\leq i,j \leq 2,
\end{equation*}
where $t^{i,j}$ denotes the entries of $\boldsymbol{T}^{-1}$. The first step of the reconstruction consists of recovering $\theta$, and hence $\boldsymbol{S}$, from the equation \cite[Eq.~(65)]{monard2012a}
\begin{equation}\label{eq:theta}
    \nabla\theta=\boldsymbol{f},
\end{equation}
with
\begin{equation*}
    \boldsymbol{f}=\frac{1}{2}\bigl(\boldsymbol{v}_{1,2}-\boldsymbol{v}_{2,1}-\boldsymbol{J}\nabla\log d\bigr),
\end{equation*}
where $\boldsymbol{J}=\begin{bmatrix}0 & -1\\ 1 & 0\end{bmatrix}$ and $d=(h_{1,1}h_{2,2}-h_{1,2}^2)^{1/2}$. Once $\theta$ is known at least one boundary point, one may integrate $\boldsymbol{f}$ along curves emanating from that point to obtain $\theta$ throughout $\Omega$. If $\theta$ is assumed known on the entire boundary, one may instead apply the divergence operator to \eqref{eq:theta} and solve the Poisson problem with Dirichlet boundary data:
\begin{equation}\label{eq:thetaBVP}
    \begin{cases}
        \Delta\theta=\nabla\cdot\boldsymbol{f} & \text{in }\Omega,\\
        \theta=\theta_{\text{true}} & \text{on }\partial\Omega.
    \end{cases}
\end{equation}

\subsection{Reconstruction of \texorpdfstring{$\sigma$}{sigma}}

The reconstruction of $\sigma$ is based on \cite[Eq.~(68)]{monard2012a}
\begin{equation}\label{eq:sigmaRec}
    \nabla(\log\sigma)=\boldsymbol{g},
\end{equation}
where
\begin{align*}
    \boldsymbol{g}&=\cos(2\theta)\boldsymbol{K}+\sin(2\theta)\boldsymbol{K},\\
    \boldsymbol{K}&=\boldsymbol{U}(\boldsymbol{v}_{1,1}-\boldsymbol{v}_{2,2})+\boldsymbol{J}\boldsymbol{U}(\boldsymbol{v}_{1,2}-\boldsymbol{v}_{2,1}),\\
    \boldsymbol{U}&=\begin{bmatrix}1 & 0\\ 0 & -1\end{bmatrix}.
\end{align*}
As in the reconstruction of $\theta$, this gradient equation can be solved either by integration along curves, if $\sigma$ is known at one boundary point, or by solving the following boundary value problem if $\sigma$ is known at the whole boundary
\begin{equation}\label{eq:sigmaBVP}
    \begin{cases}
        \Delta(\log\sigma)=\nabla\cdot\boldsymbol{g} & \text{in }\Omega,\\
        \log\sigma=\log(\sigma_{\text{true}}) & \text{on }\partial\Omega.
    \end{cases}
\end{equation}

\subsection{Non-vanishing critical points and non-vanishing Jacobian}
The following result by \cite{AlessandriniMagnanini94} gives conditions on boundary functions so that there are no critical points in the interior of the domain:

\begin{lemma}[{\cite[Thm. 2.7]{AlessandriniMagnanini94}}] \label{lem:CriticPoints}
Let \(\Omega\) be a bounded simply connected domain with Lipschitz boundary.  
Assume that \(\partial\Omega\) can be decomposed into two connected arcs \(\Gamma_1\) and \(\Gamma_2\) such that \(u|_{\partial\Omega} \in H^{1/2}(\partial\Omega)\) is nondecreasing (with respect to the arclength parameter) on \(\Gamma_1\) and nonincreasing on \(\Gamma_2\). Let $u$ be the corresponding unique solution to \eqref{eq:pde-conductivity}. Then $\nabla u(\xi)\neq 0$ for all $\xi\in \Omega$. 
\end{lemma}

In \cite{Salo2022} these conditions were extended to a pair of boundary functions $(f_1,f_2)$ so that the corresponding solutions to \eqref{eq:pde-conductivity} satisfy condition \eqref{eq:JacobiCond}. These conditions are related to the winding number of a curve $\dot{\gamma}$ around the center of $\Omega$. We denote the winding number by $\mathrm{Ind}$, so that the conditions read as follows:

\begin{lemma}[{\cite[Thm. 2.2(a)]{Salo2022}}]
\label{lem:Jacobian}
Let $\Omega$ be a bounded simply connected domain with $C^1$ boundary curve $\eta: [0,2\pi] \to \partial \Omega$, and let $\sigma \in C^{0,\alpha}(\overline{\Omega})$ satisfy $0<\alpha \leq \sigma \leq \sigma_{\text{max}}$. Let $\Gamma_1 = \eta([0,\ell])$ be a closed arc in $\partial \Omega$. Let $f_1, f_2 \in C^1(\Gamma)$ be linearly independent and let $u_i$ be the corresponding unique solution to \eqref{eq:pde-conductivity}. Assume that the curve $\gamma: [0,\ell] \to \R^2$, $\gamma(t) = (f_1(\eta(t)), f_2(\eta(t)))$ is regular and assume $\mathrm{arg}(\dot{\gamma}(t))$ is monotone. If $u_i|_{\partial \Omega}$ is continuous, and $\abs{\mathrm{Ind}(\dot{\gamma})} \leq 1$ then $\mathrm{det} [\nabla u_1(\xi) \, \nabla u_2(\xi)]\neq 0$ for all $\xi \in \Omega$.
\end{lemma}

Qualitatively, one can think about points, where the conditions $\nabla u \neq 0$ or $\det [\nabla u_1 \, \nabla u_2]\neq 0$ are violated as points, where the power density measurements $\sigma \abs{\nabla u}^2$ or $\sigma \nabla u_i \cdot \nabla u_j$ with $i,j=1,2$ do not contain information about the conductivity.




\section{Finite Element Discretization of AET} \label{sec:FEM}
In this section we briefly describe the numerical approximation to the solution of \eqref{eq:pde-conductivity} and then introduce approximate forward operation for \eqref{eq:forward-exact}.

Let us introduce test functions $w\in H^1(\Omega)$, i.e., a function with bounded first derivative. Multiplying \eqref{eq:pde-conductivity} with $w$ and integration over the domain $\Omega$ yields
\begin{equation} \label{eq:weak-u}
    \int_{\Omega} \sigma\nabla u \cdot \nabla w ~d\boldsymbol{x} = 0.
\end{equation}
To impose the boundary conditions in \eqref{eq:pde-conductivity}, we choose the lifting method \cite{quarteroni2006numerical}. Take $g_i\in H^1(\Omega)$ such that $g_i|_{\Gamma_1} =f_i $, $i=1,\dots,d$. We then introduce the auxiliary function $v_i\in H^1(\Omega)$ with a vanishing trace, i.e., $v_i|_{\Gamma_1} \equiv 0$, and define $u_i = v_i+g_i$. This reformulates \eqref{eq:weak-u} into
\begin{equation} \label{eq:weak}
    \int_{\Omega} \sigma\nabla v_i \cdot \nabla w ~d\boldsymbol{x}= - \int_{\Omega} \sigma\nabla g_i \cdot \nabla w~d\boldsymbol{x}.
\end{equation}
Note that the contribution of the boundary $\Gamma_2$ vanishes in the weak form. Now consider a triangulated discretization of the domain $\Omega$ and let $\{ \phi_i\}_{i=1}^{N_\text{FEM}}$ and $\{ \psi_i\}_{i=1}^{N_\text{FEM}}$ be first-order and zero-order Lagrange polynomials approximating $H^1(\Omega)$ and $L^2(\Omega)$. Furthermore, suppose that $\boldsymbol{v}_i, \boldsymbol{g}_i$ , $i=1,\dots,d$, and $\boldsymbol{\sigma}$ represent the vector of expansion coefficients of $v_i, g_i$ and $\sigma$ in these spaces, respectively. Now, a discrete representation of \eqref{eq:weak} takes the form
\begin{equation} \label{eq:discrete}
    \boldsymbol K \boldsymbol v_i = \boldsymbol b_i, \qquad i=1,\dots,d,
\end{equation}
where components of $\boldsymbol K\in \mathbb R^{N_{\text{FEM}} \times N_{\text{FEM}}}$ and $\boldsymbol{b}_i\in \mathbb R^{N_{\text{FEM}}}$ are
\begin{equation} \label{eq:FEM-mat-vec}
    [\boldsymbol{K}]_{m,n} = \int_\Omega \sigma \nabla \phi_m \cdot \nabla \phi_n ~d\boldsymbol{x}, \qquad [\boldsymbol b_i]_m = \int_\Omega \sigma \nabla g_i \cdot \nabla \phi_m~d\boldsymbol{x},  \qquad m,n=1,\dots,N_{\text{FEM}}.
\end{equation}
Here, with an abuse of notation, we are referring to finite element approximation $\sigma^h$ of $\sigma$ with the same symbol. Therefore, the basis functions $\{\psi_i\}_{i=1}^{N_{\text{FEM}}}$ (used to expand $\sigma$) do not explicitly appear in \eqref{eq:FEM-mat-vec}.

Now we explain how to approximate the energy density $h_{i,j}$ in \eqref{eq:energy-density}. We choose to expand $h_{i,j}$ using the zero-order Lagrange polynomials. Let $\boldsymbol{h}_{i,j}$ represent the vector of expansion coefficients in the basis $\{ \psi_i \}_{i=1}^{N_\text{FEM}}$ defined above. We can compute components of $\boldsymbol{h}$ from
\begin{equation} \label{eq:FEM-energy-density}
    [\boldsymbol h_{i,j}]_n = \int_\Omega \sigma \nabla u_i\cdot \nabla u_j \psi_n, \qquad i,j=1,\dots,d.
\end{equation}
where $u_l = g_l + v_l$, for $l=i,j$, are solutions to \eqref{eq:weak}. We remark that since we choose first-order Lagrange polynomials to expand $u_l$, then $\nabla u_l$ is piecewise constant, and hence, all terms in the right-hand-side of \eqref{eq:energy-density} are piecewise constant. We now define the approximate forward operator $\mathcal G_{i,j}$ to be the mapping $\boldsymbol{\sigma} \mapsto \boldsymbol{h}$ and define the approximate forward problem
\begin{equation} \label{eq:AET-discrete}
    \boldsymbol{y}_{i,j} = \mathcal G_{i,j} [\boldsymbol{\sigma}] + \boldsymbol{\varepsilon}_{i,j}.
\end{equation}
Here, we choose to expand noise in the FEM basis with coefficient vector $\boldsymbol{\varepsilon}_{i,j}$, yielding the approximation
\begin{equation} \label{eq:noise_FEM}
    \varepsilon_{i,j} \approx  \tau_{\text{noise}} \sum_{n=1}^{N_{\text{FEM}}} [\boldsymbol{\varepsilon}_{i,j}]_n \psi_n, 
\end{equation}
where $\tau_{\text{noise}}>0$ is a noise scaling. 

\section{Bayesian Formulation of AET} \label{sec:bayes}
In this section, we reformulate the AET inverse problem \eqref{eq:AET-discrete} within the Bayesian framework. In this setting, the unknown parameters, measurement noise, and observed data are all modeled as random variables. The solution to the inverse problem is then given by the conditional probability distribution of the unknown parameters, conditioned on the measurement data, which is referred to as the \emph{posterior distribution}. According to Bayes’ theorem, the posterior distribution is proportional to the product of the \emph{likelihood distribution}, the distribution of the data for a fixed realization of the unknown, and the \emph{prior distribution}, which encodes information about the unknown before any measurements are taken.

In the following sections, we introduce the components of this Bayesian formulation, namely the prior, likelihood, and posterior distributions for the AET problem.
\subsection{Gaussian Priors and Pushforward Measures} \label{sec:level-set}
In this section, we first recall the notion of Gaussian random fields. We then review the push-forward method for constructing probability distributions on positive smooth fields and piecewise constant fields.

Let $(\mathbb X, \langle \cdot, \cdot \rangle, \| \cdot \|)$ be a Hilbert function space and $(\mathbb X,\mathcal B(\mathbb X), \mathbb P)$, where $\mathcal B(\mathbb X)$ is the Borel $\sigma$-algebra, be a probability space defined on $\mathbb X$. We say that $X$ is an $\mathbb X$-valued Gaussian random function if for any $\mu \in \mathbb X$, the real-valued random variable $\langle X, \mu \rangle$ is a Gaussian, i.e.,  $\langle X, \mu \rangle \sim \mathcal N(m,\tau^2)$, for some $m\in \mathbb R$ and $\tau\in \mathbb R^+$.

The following lemma fully characterizes $X$ in terms of a mean function $m\in \mathbb X$ and a symmetric, trace-class and non-negative linear operator $\mathcal C:\mathbb X\to \mathbb X$. 
\begin{lemma} \cite{ibragimov2012gaussian}
Let $\mu$ be an $\mathbb X$-valued Gaussian random function, then we can find $m \in H$ and a trace-class, symmetric, and non-negative linear operator $\mathcal C:\mathbb X \to \mathbb X$, referred to as the \emph{covariance operator}, such that
    \begin{equation}
        \begin{aligned}
            \langle m , \mu \rangle &= \mathbb E \langle  X , \mu \rangle, \qquad &\forall \mu\in \mathbb X, \\
            \langle \mathcal C \mu , \eta \rangle &= \mathbb E \langle X - m , \mu \rangle \langle X - m , \eta \rangle, \qquad &\forall \mu,\eta \in \mathbb X,
        \end{aligned}
    \end{equation}
where $\mathbb E$ denotes expectation,
\[
    \mathbb E(f(X)) := \int_H f(\mu)  \ d\mathbb P(\mu).
\]
We then write $\mathcal N(m, \mathcal C) := \mathbb P \circ X^{-1}$ and $X \sim \mathcal N(m, \mathcal C)$, when referring to a Gaussian random function $X$ defined on the probability space $(\mathbb X, \mathcal B (\mathbb X), N(m, \mathcal C))$.
\end{lemma}
The following lemma recalls the \emph{Karhunen-Loève} (KL) expansion, which enables us to express $X$ in terms of the spectral decomposition of the covariance operator $\mathcal C$.

\begin{lemma} \cite{ibragimov2012gaussian} \label{thm:kl}
    Let $m$ and $\mathcal C$ be the mean and covariance operators defined above. Furthermore, let $\{e_j\}_{j=1}^{\infty}$ be the eigenfunctions and $\{ \lambda_j \}_{j=1}^{\infty}$ be the corresponding eigenvalues, sorted in decreasing order of the eigenvalues. Then, $X\sim \mathcal N(m,\mathcal C)$ if and only if $X$ has the infinite expansion
    \begin{equation} \label{eq:kl}
        X= m + \sum_{j=1}^{\infty} \sqrt{\lambda_j} X_j e_j,
    \end{equation}
    where $X_j\sim \mathcal N(0,1)$, $j \geq 1$, are independent standard normal real-valued random variables. We interpret the infinite summation as $\mathbb E \| X - m \|^2 < \infty$.
\end{lemma}
This lemma provides a practical recipe for constructing Gaussian random functions, or fields. First, choose a covariance operator $\mathcal{C}$ and compute its eigendecomposition. Then assemble the expansion in \eqref{eq:kl} and truncate the series so that a desired proportion of the variance $\mathbb{E}\|\tilde X - m\|^2$ is retained, where $\tilde X$ denotes the truncated KL expansion.

A useful class of two-dimensional covariance operators satisfying the conditions of \Cref{thm:kl} is given by \cite{dunlop2017hierarchical}
\begin{equation} \label{eq:matern-cov}
\mathcal{C} = (\tau I - \Delta)^{-\alpha},
\end{equation}
for some $\tau > 0$ and $\alpha > 1.5$. These operators are motivated by the Whittle–Mat\'ern covariance kernels \cite{whittle1954}, which allow control over both the local correlation (through $\tau$) and the regularity of sample functions (through $\alpha$). It is well known that if $X \sim \mathcal{N}(0,\mathcal{C})$ with $\mathcal{C}$ as in \eqref{eq:matern-cov}, then $X \in W^{s,2}$ for all $s < \alpha - 1$, almost surely \cite{dunlop2017hierarchical}.

To obtain an eigen-decomposition for \eqref{eq:matern-cov}, we first find an eigen-decomposition for $(\tau I - \Delta)$, to obtain $\{e_i\}_{i=1}^{N_{\text{KL}}}$ and $\{\gamma_i\}_{i=1}^{N_{\text{KL}}}$, using FEM. Then we define
\begin{equation} \label{eq:eig_inverse_relation}
    \lambda_i := \gamma_i^{-\alpha}, \qquad i=1,\dots,N_{\text{KL}},
\end{equation}
to assemble the truncated KL expansion in \eqref{eq:kl}.

In the next section, we combine Gaussian random fields with nonlinear mappings to define probability measures on the desired function spaces, namely positive smooth functions and piecewise constant functions with prescribed levels. These constructions are inspired by \cite{dunlop2017hierarchical}.
\subsection{Smooth and Positive Random Fields} \label{eq:prior-smooth}
Conductivity in the AET problem is a positive field, therefore, it is desirable to define a Bayesian prior that exhibits this property. A typical approach to consider such priors are log-Gaussian fields, where we impose a Gaussian prior on the log of the field, i.e., we consider
\begin{equation}
    X =\log \sigma \sim \mathcal N(0,\mathcal C),
\end{equation}
where $\mathcal N(0,\mathcal C)$ is a Gaussian random field with a trace-class symmetric and non-negative covariance operator $\mathcal C$, defined in \Cref{sec:level-set}. Alternatively, we can define $\sigma$ in terms of the latent variable $X$ as
\begin{equation} \label{eq:pushforward-exp}
    \sigma = F_1(X) := \exp(X), \qquad X \sim \mathcal N(0,\mathcal C).
\end{equation}

\subsection{Piecewise-Constant Random fields and Their Smooth Approximation} \label{eq:prior-piecewise-constant}
In many AET applications, the conductivity $\sigma$ exhibits a piecewise constant structure; for example, the electrical conductivity of cancerous tissue may differ significantly from the background conductivity in lung tissue \cite{farina2023histology}. To accommodate such models, we assume known foreground and background conductivities $\sigma^+ > 0$ and $\sigma^- > 0$, respectively. We then use the Heaviside function, defined by $H(\xi) = \boldsymbol{1}_{[0,\infty)}(\xi)$, where $\boldsymbol{1}$ denotes the characteristic function, to construct a two-level piecewise constant field:
\begin{equation} \label{eq:heaviside}
    \sigma = F_2(X) := \sigma^- + (\sigma^+ - \sigma^-) H(X), \qquad X \sim \mathcal{N}(0,\mathcal{C}),
\end{equation}
where $X$ is a Gaussian random field defined in \Cref{sec:level-set}.

Although this structure is a useful modeling tool, the piecewise constant nature may lead to analytical difficulties and computational inefficiencies. To address this, we approximate $F_2$ by a function with continuous spatial transitions that remains close to a piecewise constant field. This can be achieved using a sigmoid transformation:
\begin{equation}
    \sigma = F_3(X) := \sigma^- + (\sigma^+ - \sigma^-) \frac{1}{1 + \exp(-\alpha X)}, \qquad X \sim \mathcal N(0,\mathcal C),
\end{equation}
where $\alpha > 0$ controls the sharpness of the transition. As $\alpha \to \infty$, the mapping becomes increasingly steep and $F_3$ approaches $F_2$.

In all cases above, we define a probability measure on the target conductivity space $\mathbb S$, namely, piecewise smooth or piecewise constant fields, via the push-forward of $\mathbb{P}$ under $F_{\ell}$, that is, $\mathcal P :=\mathbb{P} \circ F_{p}^{-1}$ for $p = 1,2,3$. This construction defines the probability space $(\mathbb S, \mathcal B(\mathbb S), \mathcal P)$ for the conductivity $\sigma$. Note that the inverse map appears only as a measure-theoretic device for defining probability measures on the desired function spaces; it is not evaluated explicitly.


\subsection{Likelihood and Posterior} \label{sec:likelihood-posterior}
In this section, we formulate the likelihood function for the AET inverse problem. This will later be combined with the prior distribution via Bayes’ theorem to obtain the posterior distribution.

Recall from the previous section that we defined $F_p(X)$, $p = 1,2,3$ (cf. \Cref{eq:prior-smooth,eq:prior-piecewise-constant}), as the random fields representing the conductivity field $\sigma$, with prior distribution $X\sim \mathcal N(0,\mathcal C)$. We now reformulate the deterministic inverse problem \eqref{eq:AET-discrete} in a statistical setting. For each pair $i,j=1,\dots,d$, let $Y_{i,j}$ denote the random variable corresponding to the observed data $\boldsymbol y_{i,j}$, and let $E_{i,j}$ denote the random variable describing the observational noise $\boldsymbol \varepsilon_{i,j}$. The probabilistic version of the discrete AET inverse problem \eqref{eq:AET-discrete} is then
\begin{equation} \label{eq:AET-probabilistic}
    Y_{i,j} = \mathcal G_{i,j}[F_{p}(X) ] + E_{i,j}, \qquad i,j=1,\dots d,~ p = 1,2 \text{ or }3,
\end{equation}
A standard approach to defining the likelihood distribution, i.e., the distribution of the conditional random variable $Y_{i,j},|F_p(X) = \boldsymbol \sigma$, is to observe that $Y_{i,j} - \mathcal G_{i,j}[\boldsymbol{\sigma}] \sim E_{i,j}$. Therefore, the likelihood distribution is obtained by shifting the distribution of $E_{i,j}$ by $\mathcal G_{i,j}[\boldsymbol \sigma]$.

In practice, measurements are often collected pointwise, and the noise level is typically specified by the measurement device. However, in our AET formulation the random variables $Y_{i,j}$ are modeled as $L^1$, or, $L^2$ functions, for which individual pointwise measurements may not be informative. Fortunately, the choice of basis functions used to represent $h_{i,j}$ provides a natural link between pointwise data and function-valued measurements.
Recall that we approximate $h_{i,j}$ using zero-order Lagrange basis functions on a triangulated mesh, yielding the FEM expansion
\begin{equation} \label{eq:FEM-expansion-h}
h_{i,j} \approx h = \sum_{n=1}^{N_{\text{FEM}}} [\boldsymbol h]_n \psi_n,
\end{equation}
where $\boldsymbol{h}$ is the vector of expansion coefficients. Let $\xi_\star \in \Omega$ denote a node of the mesh with associated basis function $\psi_\star$. By construction, when $\xi_\star$ is a mesh node, the Lagrange basis functions satisfy
\begin{equation} \label{eq:DG-relation}
h_{i,j}(\xi_\star) = [\boldsymbol h]_\star,
\end{equation}
where $[\boldsymbol h]_\star$ is the FEM coefficient of $\boldsymbol{h}$ corresponding to $\psi_\star$ located at the mesh node $\xi_\star$ in \eqref{eq:FEM-expansion-h}. This establishes a direct correspondence between pointwise measurements at nodal locations $\xi_i$, $i=1,\dots,N_{\text{FEM}}$, and the FEM coefficients $h_{i,j}$.

Now we construct a noise model for $\boldsymbol{y}_{i,j}$ relative to $\| h_{i,j} \|_{L^1}$ or $\| h_{i,j} \|_{L^2}$. We can express the $L^2$-norm of a function $f$ approximated with FEM as $\|f\|_{L^2}^2 = \boldsymbol{f}^T \boldsymbol{M}_{\text{mass}}\boldsymbol{f }$, where $\boldsymbol{f}$ is the vector of FEM expansion coefficients and $\boldsymbol{M}_{\text{mass}}$ is the FEM mass matrix. Suppose that we have a noisy function $\tilde f$. Expanding $\| f - \tilde f \|_{L^2}$ in FEM basis yields
\begin{equation}
    \| f - \tilde f \|_{L^2} =\delta\boldsymbol{f}^T \boldsymbol{M}_{\text{mass}} \delta \boldsymbol{f},
\end{equation}
where $\delta \boldsymbol{f}$ are FEM expansion coefficients of noise $\varepsilon := f-\tilde f$. We can now get a nodal noise estimation via
\begin{equation}
    \begin{aligned}
    |[\delta \boldsymbol f]_i| &= |\boldsymbol{e}_i^T \delta\boldsymbol{f} | = | \boldsymbol{e}_i^T\boldsymbol{M}_{\text{mass}}^{-1/2} \boldsymbol{M}^{1/2}_{\text{mass}} \delta \boldsymbol{f} | = | (\boldsymbol{M}_{\text{mass}}^{-1/2}\boldsymbol{e}_i)^T \boldsymbol{M}^{1/2}_{\text{mass}} \delta \boldsymbol{f} | \\
    &\leq \| \boldsymbol{M}_{\text{mass}}^{-1/2} \boldsymbol{e}_i \|_{\ell^2} \| \boldsymbol{M}^{1/2}_{\text{mass}} \delta \boldsymbol{f} \|_{\ell^2} = \sqrt{\boldsymbol{e}_i^T \boldsymbol{M}_{\text{mass}}^{-1} \boldsymbol{e}_i } \sqrt{ \delta \boldsymbol{f}^T \boldsymbol{M}_{\text{mass}} \delta \boldsymbol{f}} \\
    &= \sqrt{ [\boldsymbol{M}_{\text{mass}}^{-1}]_{ii} } \| \varepsilon \|_{L^2}.
    \end{aligned}
\end{equation}
Here, $\boldsymbol{e}_i$ is the $i$th column of unity matrix, and the in-equality is due to application of the Cauchy-Schwartz inequality to the inner-product of 2 vectors. Hence, a prescribed $L^2$ noise level directly yields admissible nodal noise bounds, ensuring consistency between functional and nodal noise models. Conversely,
\begin{equation}
    \| \varepsilon \|^2_{L^2}
    = \delta \boldsymbol{f}^T \boldsymbol{M}_{\text{mass}} \delta \boldsymbol{f}
    \le \lambda_{\text{max}}(\boldsymbol{M}_{\text{mass}})\,
    \| \delta \boldsymbol{f} \|_{\ell^2}^2,
\end{equation}
where $\lambda_{\text{max}}(\boldsymbol{M}_{\text{mass}})$ denotes the largest eigenvalue of the mass matrix, providing an upper bound for the global functional noise in terms of the nodal noise.


We can extract a similar relation when the noise is with respect to the $L^1$-norm of the measurement. Suppose $[\delta \boldsymbol{f}]_i = \varepsilon(\xi_i)$ is the noise at the $i$th FEM mesh node $\xi_i$. Discrete Riesz representation (in FEM space of $V_h$=span$\{\psi_j\}$) tells us that there is a unique element $\phi_i\in V_h$, such that $\langle\phi_i,f\rangle_{L^2} = f(\xi_i)$, for any FEM function $f$. By expanding this function in FEM basis we obtain
\begin{equation}
    \boldsymbol{e}_i^T \boldsymbol{f} = [\boldsymbol{f}]_i = f(\xi_i) = \langle\phi_i,f\rangle_{L^2} = \boldsymbol{\phi}_i^T \boldsymbol{M} \boldsymbol{f}.
\end{equation}
Here, $\boldsymbol{\phi}_i$ is FEM expansion coefficients of $\phi_i$. Since this relation holds for all vectors $\boldsymbol{f}$ then we must have $\boldsymbol{e}^T_i = \boldsymbol{\phi}_i^T \boldsymbol{M}$ and thus $\phi_i = \sum_j (\boldsymbol{M}^{-1})_{ji}\psi_j$. We now apply this to nodal noise to obtain
\begin{equation}
    \varepsilon(\xi_i) = \langle\phi_i,\varepsilon\rangle_{L^2} \leq \| \sum_j (\boldsymbol{M}^{-1})_{ji}\psi_j \|_{L^{\infty}} \| \varepsilon \|_{L^1} \leq \left( \sum_j |(\boldsymbol{M}^{-1})_{ji}| \right) \| \varepsilon \|_{L^1}.
\end{equation}
Here we applied the H\"older inequality with $L^1$-$L^\infty$ duality pairing. We can also bound the global functional noise in terms of the $L^1$ nodal noise as
\begin{equation}
    \|\varepsilon\|_{L^1}
    \le
    \sum_{n=1}^{N_{\text{FEM}}}
    |[\delta \boldsymbol{f}]_n|
    \int_{\Omega} |\psi_n(\xi)|\, d\xi,
\end{equation}
which follows from expanding $\varepsilon$ in the FEM basis and applying the triangle inequality. This bound admits a similar interpretation as in the $L^2$ noise case. In the remainder of this article we assume noise is specified in a functional norm.

The likelihood function considered here is
\begin{equation} \label{eq:likelihood}
    L(\boldsymbol{y}_{i,j};\boldsymbol{x} ) \propto \exp\left( - \sum_{i,j=1}^{d} \frac{\| \mathcal G_{i,j}[F_p(\boldsymbol{x})] - \boldsymbol{y}_{i,j} \|_{L^\ell}^\ell}{\ell (\tau_{i,j}^{\text{noise}})^\ell} \right)
\end{equation}
where $\ell$ stands for 1 or 2, $\boldsymbol{x}$ collects KL expansion coefficients defined in \Cref{sec:level-set}, and $\tau^{\text{noise}}_{i,j}$ indicates the scale of noise according to noise-level $d_{\text{noise}}$ relative to $L^\ell$-norm, i.e.,
\begin{equation} \label{eq:noise-level}
    \tau^{\text{noise}}_{i,j} := d_{\text{noise}} \frac{ \|  y^{\text{noise-free}}_{i,j} \|_{L^\ell}}{\| y^{\text{noisy}}_{i,j} - y^{\text{noise-free}}_{i,j} \|_{L^\ell} }.
\end{equation}
Here, $0<d_{\text{noise}}$ and the superposition of the exponents in \eqref{eq:likelihood} is due to the assumption that the measurement random variables $Y_{i,j}$ are independent random variables.

We also define the negative log-likelihood function to be
\begin{equation} \label{eq:negative-log-likelihood}
    \Phi(\boldsymbol y_{i,j}; \boldsymbol{x} ) := \sum_{i,j=1}^{d} \frac{\| \mathcal G_{i,j}[ F_p( \boldsymbol{x}) ] - \boldsymbol y_{i,j} \|_{L^\ell}^\ell}{\ell (\tau_{i,j}^{\text{noise}})^\ell},
\end{equation}
where, with an abuse of notation, we refer to $\boldsymbol y_{i,j}$ to be the collection all measurements for $i,j=1,\dots,d$.

\section{Wellposedness}
\label{sec:Wellposedness}
In this section, we show that the likelihood defined in \eqref{eq:likelihood} fits within the well-posedness framework of \cite{stuart2010inverse}. We first establish the result for the $L^1$ norm appearing in \eqref{eq:likelihood}. We then show that additional regularity in the boundary input, which ensures that $h_{i,j}$ belongs to $L^2$, leads to stronger bounds in the well-posedness analysis. Throughout this section, we assume that $\sigma$ is a random field with prior measure $\mathcal{P} = \mathbb{P} \circ X^{-1} \circ F_p^{-1}$, for $p=1,2,3$, constructed as in \Cref{eq:prior-smooth,eq:prior-piecewise-constant}. In order to establish the well-posedness results we first show Lipschitz continuity of the map $\sigma \mapsto h_{i,j}$ and under which regularity assumptions on $\partial \Omega$, $\sigma$ and the boundary inputs $f_i$ we obtain $h_{i,j}\in L^1$ and $h_{i,j}\in L^2$ respectively for $i,j=1,...,d$.

\begin{lemma}[Lipschitz continuity of the map $\sigma \mapsto h_{i,j}$ from $L^{\infty}$ to $L^1$]
    Consider $\sigma_1,\sigma_2\in L^{\infty}$ such that $0<\sigma_{\text{min}}\leq \sigma_1,\sigma_2 \leq \sigma_{\text{max}}$ and associated solutions $u_i^{(1)}$ and $u_i^{(2)}$ corresponding to boundary conditions $g_i$ for $i=1,...,d$ imposed:
    \begin{equation}\label{eq:sig12}
    \begin{cases}
                    \text{div}( \sigma_k \nabla u_i^{(k)}) = 0 & \text{in }\Omega\\
                    u_i^{(k)} = g_i & \text{on } \partial \Omega,
    \end{cases}
\end{equation}
    The associated measurements corresponding to $\sigma_k$ are on the form $h_{i,j}(\sigma_k):=\sigma_k \nabla u_i^{(k)}\cdot \nabla u_j^{(k)}$. The map $\sigma \mapsto h_{i,j}(\sigma)$ is Lipschitz continuous from $L^{\infty}$ to $L^1$:
    \begin{equation}\label{eq:Lipschitz}
        \norm{h_{i,j}(\sigma_1)-h_{i,j}(\sigma_2)}_{L^1}\leq C \norm{\sigma_1-\sigma_2}_{L^{\infty}}
    \end{equation}
\end{lemma}
\begin{proof}
The function $w=u_i^{(1)}-u_i^{(2)}$ satisfies the boundary value problem
\begin{equation*}
    \begin{cases}
                    \text{div}( \sigma_1 \nabla w) = \text{div} \lp(\sigma_1-\sigma_2)\nabla u_i^{(2)}\rp & \text{in }\Omega\\
                    w = 0 & \text{on } \partial \Omega.
    \end{cases}
\end{equation*}
With the estimate \cite[p.524]{salsa2016a} one obtains the following continuity estimate for $w$:
\begin{align}
    \norm{w}_{H^1(\Omega)}=\norm{u_i^{(1)}-u_i^{(2)}}_{H^1}&\leq C(n,\Omega,\sigma_{\text{max}},\sigma_{\text{min}})\norm{(\sigma_1-\sigma_2)\nabla u_i^{(2)}}_{L^2}\nonumber\\
    &\leq C(n,\Omega,\sigma_{\text{max}},\sigma_{\text{min}}) \norm{g_i}_{H^{\frac{1}{2}}}\norm{\sigma_1-\sigma_2}_{L^\infty}\label{eq:west}
\end{align}
We now investigate continuity for $h_{i,j}$:
    \begin{equation} \label{eq:Lipschtitz2}
\begin{aligned}
    \norm{h_{i,j}(\sigma_1)-h_{i,j}(\sigma_2)}_{L^1}&=\norm{\sigma_1 \nabla u_i^{(1)}\cdot \nabla u_j^{(1)}-\sigma_2 \nabla u_i^{(2)}\cdot \nabla u_j^{(2)}}_{L^1}\\
    &=\norm{(\sigma_1-\sigma_2) \nabla u_i^{(1)}\cdot \nabla u_j^{(1)}+\sigma_2 \lp\nabla u_i^{(1)}\cdot \nabla u_j^{(1)}-\nabla u_i^{(2)}\cdot \nabla u_j^{(2)}\rp}_{L^1}\\
    &\leq \norm{\sigma_1-\sigma_2}_{L^\infty}\norm{\nabla u_i^{(1)}\cdot \nabla u_j^{(1)}}_{L^1}\\
    &\hspace{6mm}+\norm{\sigma_2}_{L^{\infty}} \norm{\nabla u_i^{(1)}\cdot \nabla u_j^{(1)}-\nabla u_i^{(2)}\cdot \nabla u_j^{(2)}}_{L^1(\Omega)}\\
    &\leq \norm{\sigma_1-\sigma_2}_{L^\infty}\norm{\nabla u_i^{(1)}}_{L^2}\norm{\nabla u_j^{(1)}}_{L^2}\\
    &\hspace{6mm}+\norm{\sigma_2}_{L^{\infty}} \norm{\nabla u_i^{(1)}\cdot \nabla u_j^{(1)}-\nabla u_i^{(2)}\cdot \nabla u_j^{(2)}}_{L^1(\Omega)}\\
\end{aligned}
\end{equation}

Observe that 
\begin{align*}
    \nabla u_i^{(1)}\cdot \nabla u_j^{(1)}-\nabla u_i^{(2)}\cdot \nabla u_j^{(2)}&=(\nabla u_i^{(1)}-\nabla u_i^{(2)})\cdot (\nabla u_j^{(1)}+\nabla u_j^{(2)})\\
    &\hspace{6mm}+(\nabla u_i^{(1)}+\nabla u_i^{(2)})\cdot (-\nabla u_j^{(1)}+\nabla u_j^{(2)})
\end{align*}
Using Hölder's inequality this implies
\begin{align*}
    \norm{\nabla u_i^{(1)}\cdot \nabla u_j^{(1)}-\nabla u_i^{(2)}\cdot \nabla u_j^{(2)}}_{L^1(\Omega)}&\leq\norm{\nabla u_i^{(1)}-\nabla u_i^{(2)}}_{L^2}\lp \norm{\nabla u_j^{(1)}}_{L^2}+\norm{\nabla u_j^{(2)}}_{L^2}\rp\\
    &\hspace{6mm}+\norm{\nabla u_j^{(2)}-\nabla u_j^{(1)}}_{L^2}\lp \norm{\nabla u_i^{(1)}}_{L^2}+\norm{\nabla u_i^{(2)}}_{L^2}\rp
\end{align*}
Using the estimate the estimate \cite[p.524]{salsa2016a} and the estimate \eqref{eq:west} yields
\begin{align*}
    \norm{\nabla u_i^{(1)}\cdot \nabla u_j^{(1)}-\nabla u_i^{(2)}\cdot \nabla u_j^{(2)}}_{L^1(\Omega)}&\leq C \norm{\sigma_1-\sigma_2}_{L^{\infty}}\lp \norm{g_j}_{H^{\frac{1}{2}}}+\norm{g_j}_{H^{\frac{1}{2}}}\rp\\
    &\hspace{6mm}+C \norm{\sigma_1-\sigma_2}_{L^{\infty}}\lp \norm{g_i}_{H^{\frac{1}{2}}}+\norm{g_i}_{H^{\frac{1}{2}}}\rp
\end{align*}

Inserting this in \eqref{eq:Lipschtitz2} and using the estimate \cite[p.524]{salsa2016a} once more we obtain the desired estimate:

\begin{equation*}
\begin{aligned}
    \norm{h_{i,j}(\sigma_1)-h_{i,j}(\sigma_2)}_{L^1}&\leq C \norm{\sigma_1-\sigma_2}_{L^{\infty}}
\end{aligned}
\end{equation*}
\end{proof}

\begin{lemma}[Regularity of $h_{i,j}$]\mbox{}\label{lem:regH}
\vspace*{2mm}
\vspace*{-\baselineskip}
    \begin{itemize}
        \item General case: Let $0<\sigma_{\text{min}} \leq \sigma \in L^{\infty}(\Omega)$ and $\Omega$ have Lipschitz boundary. Let $f_i$ for $i=1,...,d$ be such that $u\vert_{\partial \Omega} \in H^{1/2}(\partial \Omega)$ then $h_{i,j}$ is in $L^1$ for $i,j=1,...,d$.
        \item Smooth case: Let $\sigma \in H^s(\Omega)$ for $s>1$ and $\Omega$ have Lipschitz boundary. Let $f_i$ for $i=1,...,d$ be such that $u\vert_{\partial \Omega} \in H^{s+\frac{1}{2}}(\partial \Omega)$ then $h_{i,j}$ is in $L^2$ for $i,j=1,...,d$ by \cite[Thm 5.23]{adams1975a}.
    \end{itemize}
\end{lemma}

\begin{remark}[Sufficient conditions for Lemma \ref{lem:regH} in limited view]
Assume that $\sigma$ and $\Omega$ are as in Lemma \ref{lem:regH}. The following construction of a boundary function $f$ in limited view aligns with Lemma \ref{lem:regH}:
    \begin{itemize}
        \item General case: If the boundary function $f$ on $\Gamma_1$ extends continuously to 0 along $\partial \Omega$, it is piecewise differentiable and an $L^2$ function along $\Gamma_1$ then the corresponding solution $u$ satisfies $u\vert_{\partial \Omega} \in H^{\frac{1}{2}}(\partial \Omega)$. If the boundary function does not extend continuously to 0 along $\partial \Omega$ then the corresponding solution is in a weighted Sobolev space as analyzed in \cite{Salo2022}. 
        \item Smooth case: If the boundary function $f \in C^{\infty}(\Gamma)$ and extends smoothly to 0 along $\partial \Omega$ then $u\vert_{\partial \Omega} \in H^{s+\frac{1}{2}}(\partial \Omega)$ for $s>0$. 
    \end{itemize}
\end{remark}

\begin{theorem} \label{thm:wellposedness-L1}
    Let $(\mathbb S, \mathcal B(\mathbb S), \mathcal P )$, $p=1,2$ or $3$ be the probability space associated with the prior measure introduced in \Cref{sec:level-set}, and $\Phi(\boldsymbol y_{i,j};\boldsymbol{\sigma})$, $i,j=1,\dots,d$ be the negative log likelihood defined in \eqref{eq:likelihood} with the $L^1$ norm. Under the assumptions that $0< \sigma_{\text{min}} < \sigma <\sigma_{\text{max}} < \infty $ and $g_i,g_j \in H^{\frac{1}{2}}(\partial \Omega)$ we have
    \begin{enumerate}
        \item For any fixed and $L^1$-bounded measurement $\boldsymbol{y}_{i,j}$,  $L(\boldsymbol{y}_{i,j}; \cdot)$ is Borel measurable.
        \item For any fixed and bounded measurements $\| \boldsymbol{y}_{i,j} \|_{L^1}$, $\Phi(\boldsymbol{y}_{i,j}, \cdot)$ is $\mathcal N(m, \mathcal C)$-a.s. continuous.

        \item For any fixed $\sigma$ and measurements $\boldsymbol{y}^1_{i,j}$ and $\boldsymbol{y}^2_{i,j}$ with the condition\par $\text{max}\{ \| \boldsymbol{y}^1_{i,j} \|_L^1 , \|\boldsymbol{y}^2_{i,j} \|_{L^1} \} < r\in \mathbb R^+$, and $\sigma_{\text{max}}\in \mathbb R^+$, such that $\sigma < \sigma_{\text{max}}$ we can find $C$, depending on $r$ and $\sigma_{\text{max}}$ such that
        \begin{equation}
            | \Phi(\boldsymbol{y}^1_{i,j}, \sigma) - \Phi(\boldsymbol{y}^2_{i,j},\sigma) | \leq C(r,\sigma_{\text{max}}) \sum_{i,j=1}^d \| \boldsymbol{y}^1_{i,j} - \boldsymbol{y}^2_{i,j} \|_{L^1}.
        \end{equation}
    \end{enumerate}
\end{theorem}
\begin{proof}
    In this proof, we assume that the negative log-likelihood \eqref{eq:negative-log-likelihood} contains only a single term in the summation and therefore omit the summation sign. This assumption is made solely for notational simplicity, and the extension of the result to the full summation is straightforward.
    \begin{enumerate}
        \item By \eqref{eq:Lipschitz}, the forward operator
        \[
        H : (S_1,\|\cdot\|_{L^\infty}) \to (L^1,\|\cdot\|_{L^1})
        \]
        is (locally) Lipschitz continuous, and hence continuous. Therefore $H$ is Borel measurable. For fixed $\boldsymbol y_{i,j}\in L^1$, define the translation operator
        \[
        T_{\boldsymbol y_{i,j}} : L^1 \to L^1, 
        \qquad
        T_{\boldsymbol y_{i,j}}(f) = \boldsymbol y_{i,j} - f .
        \]
        The map $T_{\boldsymbol y_{i,j}}$ is continuous and hence Borel measurable. Moreover, the norm map
        \[
        N : L^1 \to \mathbb{R}, 
        \qquad
        N(f) = \|f\|_{L^1},
        \]
        is continuous and therefore Borel measurable. Since the negative log-likelihood \eqref{eq:likelihood} is a composition of Borel measurable mappings, it is itself Borel measurable.

        \item Let $\boldsymbol{y}_{i,j}\in L^1$ be fixed and take $\sigma_1,\sigma_2$ with $\mathcal N(0,\mathcal C)$ probability, then
        \begin{align*}
        |\Phi(\boldsymbol{y}_{i,j};\sigma_1)-\Phi(\boldsymbol{y}_{i,j};\sigma_2)|
        &=| \sum_{i,j=1}^d\left(  \|\boldsymbol{y}_{i,j}-h_{i,j}(\sigma_1)\|_{L^1}-\|\boldsymbol{y}_{i,j}-h_{i,j}(\sigma_2)\|_{L^1}\right)| \\
        &\leq \sum_{i,j=1}^d \|h_{i,j}(\sigma_1)-h_{i,j}(\sigma_2)\|_{L^1} \\
        &\leq C\|\sigma_1-\sigma_2\|_{L^\infty},
        \end{align*}
        where we used the inequality $|\|a\|-\|b\||\le \|a-b\|$ for the $L^1$ norm and Lipchitz condition \eqref{eq:Lipschitz} in the last step.

        \item It is sufficient to show the result for a single set of $i,j$ It follows
            \begin{equation}
        \begin{aligned}
        |\Phi( \boldsymbol{y}^1_{i,j} ; \sigma ) - \Phi( \boldsymbol{y}^2_{i,j} ; \sigma ) | & = | \sum_{i,j=1}^d \left( \| \boldsymbol{y}^1_{i,j} - h_{i,j}(\sigma) \|_{L^1} - \| \boldsymbol{y}^2_{i,j} - h_{i,j}(\sigma) \|_{L^1}  \right)  | \\
        & \leq \sum_{i,j=1}^d \| \boldsymbol{y}^1_{i,j} - \boldsymbol{y}^2_{i,j}  \|_{L^1}  \\
        \end{aligned}
    \end{equation}
    where we used reversed triangle inequality in the last step.
    \end{enumerate}
\end{proof}

\begin{theorem} \label{thm:wellposedness-L2}
    Let $(\mathbb S, \mathcal B(\mathbb S), \mathcal P )$, $p=1,2$ or $3$ be the probability space associated with the prior measure introduced in \Cref{sec:level-set}, and $\Phi(Y_{i,j};\boldsymbol{\sigma})$, $i,j=1,\dots,d$ be the negative log likelihood defined in \eqref{eq:likelihood} with the $L^2$ norm. Under the assumptions that $0< \sigma_{\text{min}} < \sigma < \sigma_{\text{max}} < \infty$ and $g_i,g_j \in H^{s+\frac{1}{2}}(\partial \Omega)$ with $s>1$ we have
    \begin{enumerate}
        \item For any fixed and $L^2$-bounded measurement $\boldsymbol{y}_{i,j}$,  $L(\boldsymbol{y}_{i,j}; \cdot)$ is Borel measurable.
        \item For any fixed and bounded measurements $\| \boldsymbol{y}_{i,j} \|_{L^2}$, $\Phi(\boldsymbol{y}_{i,j}, \cdot)$ is $\mathcal N(0, \mathcal C)$-a.s. continuous.

        \item For any fixed $\sigma$ and measurements $\boldsymbol{y}^1_{i,j}$ and $\boldsymbol{y}^2_{i,j}$ with the condition\par $\text{max}\{ \| \boldsymbol{y}^1_{i,j} \|_L^2 , \|\boldsymbol{y}^2_{i,j} \|_{L^2} \} < r\in \mathbb R^+$, and $\sigma_{\text{max}}\in \mathbb R^+$, such that $\sigma < \sigma_{\text{max}}$ we can find $C$, depending on $r$ and $\sigma_{\text{max}}$ such that
        \begin{equation}
            | \Phi(\boldsymbol{y}^1_{i,j}, \sigma) - \Phi(\boldsymbol{y}^2_{i,j},\sigma) | \leq C(r,\sigma_{\text{max}}) \sum_{i,j=1}^d \| \boldsymbol{y}^1_{i,j} - \boldsymbol{y}^2_{i,j} \|_{L^2}.
        \end{equation}
    \end{enumerate}
\end{theorem}
\begin{proof}
The proofs for parts 1 and 2 can be carried out similarly to \Cref{thm:wellposedness-L1}. To show part 3, we have
    \begin{equation}
        \begin{aligned}
        |\Phi(\boldsymbol{y}^1_{i,j}, \sigma) - \Phi(\boldsymbol{y}^2_{i,j},\sigma) | &\leq \frac{1}{2} |\sum_{i,j=1}^d \langle \boldsymbol{y}^1_{i,j} + \boldsymbol{y}^2_{i,j}  - 2 h_{i,j}(\sigma), \boldsymbol{y}^1_{i,j} - \boldsymbol{y}^2_{i,j}  \rangle_{L^2}| \\
        & \leq \frac{1}{2} \sum_{i,j=1}^d ( \| \boldsymbol{y}^1_{i,j} \|_{L^2} + \| \boldsymbol{y}^2_{i,j} \|_{L^2} + 2\| h_{i,j}(\sigma) \|_{L^2} ) \\
        & \hspace{5cm} \|  \boldsymbol{y}^1_{i,j} - \boldsymbol{y}^2_{i,j}\|_{L^2} \\
        & \leq (r + \| h_{i,j}(\sigma) \|_{L^2} ) \sum_{i,j=1}^d \| \boldsymbol{y}^1_{i,j}  - \boldsymbol{y}^2_{i,j}  \|_{L^2}.
        \end{aligned}
    \end{equation}
\end{proof}

\subsection{Posterior}
In this section, we apply Bayes' theorem to define the posterior distribution of the conductivity field, that is, the conditional distribution of $\sigma$ given the measurements $\boldsymbol{y}_{i,j}$, $i,j=1,\dots,d$. We then establish the existence of this posterior distribution.

\begin{theorem}
Let $(\mathbb S, \mathcal B(\mathbb S), \mathcal P )$, $p=1,2$ or $3$ be the probability space defined in \Cref{sec:level-set} associated with the prior distribution $\mathcal P$ for conductivity $\sigma$. Furthermore, suppose that $\Phi$ is the negative log-likelihood defined in \Cref{eq:negative-log-likelihood}. Then the posterior distribution $\mathcal P^{\text{post}}$ (the conditional probability measure of the conductivity given measurement) is absolutely continuous with respect to the prior distribution $\mathcal P$, i.e., $\mathcal P^{\text{post}} \ll \mathcal P$, and is expressed as the Radon-Nikodym derivative
\begin{equation} \label{eq:Bayes}
    \frac{d \mathcal P^{\text{post}} }{d \mathcal P} (\sigma) = \frac{1}{Z} \exp(-\Phi(\boldsymbol{y}_{i,j}; \sigma )),
\end{equation}
with normalization constant
\begin{equation}
    Z = \int_{\mathbb S}  \exp(-\Phi(\boldsymbol{y}_{i,j}; \sigma ) )  \mathcal P(d\sigma ).
\end{equation}
Furthermore, for two sets of measurements $\{\boldsymbol y_{i,j}^1\}$ and $\{\boldsymbol y_{i,j}^2\}$, $i,j=1,\dots,d$, with\par \noindent $\max \{ \| \boldsymbol y_{i,j}^1 \|_{L^\ell} , \| \boldsymbol y_{i,j}^2 \|_{L^\ell} \}_{i,j=1}^{d}<r$, $\ell =1$ or $2$, and $r>0$, there is $C>0$ independent of $\sigma$, such that
\begin{equation}
    d_{\text{Hell}}( \mathcal P^{\text{post}}_{\sigma|\boldsymbol{y}^1_{i,j}},\mathcal P^{\text{post}}_{\sigma|\boldsymbol{y}^2_{i,j}} ) \leq C \sum_{i,j=1}^d \| \boldsymbol y^1_{i,j} -\boldsymbol y^2_{i,j} \|_\ell.
\end{equation}
where, $d_{\text{Hell}}(\cdot,\cdot)$ is the Hellinger distance between probability measures \cite{le2000asymptotics}.
\end{theorem}
\begin{proof}
Bayes' theorem tells us that the relation between the posterior and the prior measure follows \eqref{eq:Bayes}. Therefore, to show that $\mathcal P^{\text{post}} \ll \mathcal P$ we show that the right-hand-side of \eqref{eq:Bayes} is well-defined, i.e.  $\Phi$ is measurable, and $Z$ is finite and positive.

In \Cref{thm:wellposedness-L1} we showed that $\Phi(\boldsymbol{y}_{i,j}; \cdot)$ and $\Phi(\cdot; \sigma)$, are a.s. continuous, and locally Lipschitz, respectively. Therefore, the mapping $\Phi(\cdot;\cdot)$ is jointly continuous, $\mathcal P$-a.s., and therefore, it is $\mathcal P$-measurable.

To show that $Z>0$ recall that $\Phi$ is bounded from above according to the first statement in \Cref{thm:wellposedness-L1}. Therefore,
\begin{equation}
    \int_{\mathbb S} \exp(-\Phi(\boldsymbol y_{i,j};\sigma)\mathcal P(d\sigma)\geq \exp(-\tilde C) \int_{\mathbb S} \mathcal P(d\sigma) = \exp(-\tilde C) > 0.    
\end{equation}
Showing that $Z$ is bounded from above is yielded by the fact that $\exp(-\Phi(\boldsymbol {y}_{i,j};\sigma) )\leq 1$. Therefore, Bayes' theorem applies and $\mathcal P^{\text{post}} \ll \mathcal P$. Hellinger well-posedness can be derived identically to Theorem 2.2 in \cite{iglesias2016bayesian}.
\end{proof}

\subsection{Posterior Exploration via MCMC} \label{sec:mcmc}

To characterize uncertainty in the reconstructed conductivity field we draw samples from the posterior distribution, i.e., we approximate the posterior distribution with a discrete set of samples. In Bayesian inverse problems this is typically achieved using Markov chain Monte Carlo (MCMC) methods \cite{mcbook}, which construct a Markov chain whose invariant distribution coincides with the posterior. We can write the discrete posterior distribution for the AET problem as
\begin{equation} \label{eq:discrete-posterior-density}
    \pi(\boldsymbol{x}) \propto \exp\left( - \sum_{i,j=1}^{d} \frac{\| \mathcal G_{i,j}[F_p(\boldsymbol{x})] - \boldsymbol{y}_{i,j} \|_{L^\ell}^\ell}{\ell (\tau_{i,j}^{\text{noise}})^\ell} - \frac{1}{2}\|\boldsymbol{x}\|^2 \right), \qquad p=1,2,\text{ or }3,
\end{equation}
where the posterior is written with respect to latent KL expansion coefficients. A classical approach to construct such Markov chain is the Metropolis-Hastings (MH) algorithm \cite{kaipio2005statistical}. Given a current state $\boldsymbol{x}^{(k)}$, a proposal $\boldsymbol{x}^\ast$ is drawn from a proposal distribution $q(\boldsymbol{x}^{(k)},\cdot)$ and accepted with probability
\[
\alpha(\boldsymbol{x}^{(k)},\boldsymbol{x}^\ast)
=
\min\!\left(
1,
\frac{\pi(\boldsymbol{x}^\ast)\, q(\boldsymbol{x}^\ast,\boldsymbol{x}^{(k)})}
     {\pi(\boldsymbol{x}^{(k)})\, q(\boldsymbol{x}^{(k)},\boldsymbol{x}^\ast)}
\right),
\]
If the proposal is rejected, the Markov chain remains at the current state, ensuring that the chain preserves the target distribution.

In function-space inverse problems, standard random-walk proposals may deteriorate as the discretization is refined. To avoid this issue we employ the preconditioned Crank--Nicolson (pCN) algorithm \cite{cotter2013mcmc}, which is designed to preserve a prior Gaussian measure and therefore remains stable under mesh refinement. Given the current state $\boldsymbol{x}^{(k)}$, a proposal is generated as
\[
\boldsymbol{x}^\ast
=
\sqrt{1-\beta^2}\,\boldsymbol{x}^{(k)}
+
\beta\,\boldsymbol{z},
\qquad
\boldsymbol{z} \sim \mathcal N(0,\mathcal C),
\]
where $\beta\in(0,1)$ controls the proposal step size. Because the proposal is prior-preserving, the Metropolis--Hastings acceptance probability simplifies to
\[
\alpha(\boldsymbol{x}^{(k)},\boldsymbol{x}^\ast)
=
\min\!\left(
1,
\exp\!\big(
-\Phi(\boldsymbol{y}_{i,j};\boldsymbol{x}^\ast)
+
\Phi(\boldsymbol{y}_{i,j};\boldsymbol{x}^{(k)} )
\big)
\right),
\]
Consequently, the acceptance ratio depends only on the likelihood, making the pCN algorithm particularly convenient for sampling posterior measures defined on function spaces.

Once samples $\{\boldsymbol{x}^{(k)}\}_{k=1}^{N_{\text{sample}}}$ are generated, we use $F_{p}$, for $p=1,2$ or $3$ to construct their corresponding conductivity samples $\{\sigma^{(k)} :=F_p(\boldsymbol{x}^{(k)})\}_{k=1}^{N_{\text{sample}}}$. Posterior statistics can then be approximated using ergodic averages. In particular, the posterior mean conductivity is estimated by
\[
\mathbb E(\sigma) \approx \bar \sigma =
\frac{1}{N_{\text{sample}}}
\sum_{k=1}^{N_{\text{sample}}}
\sigma^{(k)},
\]
which converges to the posterior expectation as $N\to\infty$ under standard ergodicity assumptions. Similarly, uncertainty in the reconstruction can be quantified through the empirical standard deviation,
\[
\tau (\xi)
=
\left(
\frac{1}{N-1}
\sum_{k=1}^{N}
\big(\sigma^{(k)}(\xi)-\bar{\sigma}_N(\xi)\big)^2
\right)^{1/2},
\]
which provides a point-wise measure of posterior variability.

Many other sampling algorithms have been proposed for Bayesian inverse problems, including the unadjusted Langevin algorithm (ULA), the Metropolis-adjusted Langevin algorithm (MALA), and Hamiltonian Monte Carlo (HMC), which exploit gradient information of the posterior to improve sample quality \cite{ dalalyan2017theoretical,dalalyan2017theoretical,roberts1996exponential,duane1987hybrid}. In the numerical experiments presented in this paper, however, the pCN method performs remarkably well, allowing relatively large step sizes and effective exploration of the posterior distribution. This behavior contrasts similar Bayesian sampling for the EIT, where pCN often requires very small step sizes, leading to inefficient sampling. Consequently, for the AET inverse problem considered here, the simplicity and robustness of the pCN algorithm provide a practical advantage.

\section{Numerical Examples}
\label{sec:Numerics}

In this section, we examine the numerical behavior of the AET problem under different noise models and limited-view configurations, and compare the results with the deterministic reconstruction method introduced in \Cref{sec:deterministic}. For the deterministic reconstruction method we use measurements $h_{i,j}$ with $1\leq i,j \leq 2$ corresponding to two boundary functions $f_1$ and $f_2$ imposed to \eqref{eq:pde-conductivity}, while for the Bayesian formulation we restrict ourselves to only use the measurement $h_{1,1}$ corresponding to $f_1$. In the first subsection we select the likelihood function and prior distribution that is most suitable for limited view AET.

\subsection{Bayesian reconstructions}
In this section we summarize the discretization and implementation details of the AET problem and present the results obtained from the Bayesian formulation. We assess the performance of the proposed method under various noise levels and limited-view configurations. In the following section, we compare the Bayesian reconstructions with those obtained using a deterministic approach and highlight the role of uncertainty quantification in identifying trustworthy regions of the reconstruction.

For the FEM implementation of the AET forward operator, we consider $\Omega$ to be the unit disk and discretize it using a regular but unstructured triangulated mesh with 7,651 degrees of freedom, corresponding to a cell size of approximately $h_{\text{mesh}} = 0.02$. First-order Lagrangian elements are used to discretize both scalar-valued functions (e.g., $u$ and $w$) and vector-valued functions (e.g., $h_{i,j}$) on this mesh.

To obtain a noise-free observation vector, we record the values of $\boldsymbol{h}$, i.e., FEM expansion coefficients of $h_{1,1}$, inside $\Omega$ on a finer regular mesh with approximately 30,000 degrees of freedom. The two meshes are chosen to be nested so that the observational nodal values corresponding to the coarse mesh can be extracted directly from the finer mesh without interpolation. This choice avoids systematic interpolation bias.

To test the method, we construct an out-of-prior phantom consisting of several inclusions of varying sizes. For numerical stability, the phantom is slightly smoothed. To investigate limited-view cases, we apply boundary inputs as shown in \Cref{fig:boundary_input_smooth} and follow the forward procedure described in \Cref{sec:FEM} to evaluate the corresponding measurement functions. The true phantom, together with the noise-free measurement $h_{1,1}$ under various limited-view angle configurations, are shown in \Cref{fig:signals}. In each experiment, the extent and location of the boundary input are indicated by a red curve. 

The chosen boundary functions $f_1$ in \Cref{fig:boundary_input_smooth} are smooth and extend smoothly to zero along $\partial \Omega$. In particular, they satisfy $g = u|_{\partial \Omega} \in H^{s+\frac{1}{2}}(\partial \Omega)$, so that Lemma \ref{lem:regH} applies in both the general and the smooth case, implying that the corresponding power densities belong to $L^1$ and $L^2$. Additionally, the functions are chosen in accordance with Lemma \ref{lem:CriticPoints}, as $u_1|_{\partial \Omega}$ can be decomposed into a nondecreasing and a nonincreasing function along $\partial \Omega$. This guarantees that the noise-free generated power density satisfies $h_{1,1} = \sigma |\nabla u_1(\xi)| \neq 0$ for all $\xi \in \Omega$. The boundary functions $f_2$ shown in \Cref{fig:boundary_input_sharp} are used solely for the analytical reconstruction procedure and are chosen in accordance with Lemmas \ref{lem:CriticPoints} and \ref{lem:Jacobian} such that both $h_{2,2} = \sigma |\nabla u_2(\xi)| \neq 0$ and $\det(\boldsymbol{H}) \neq 0$ in the noise-free case.


\begin{figure}
    \centering
    \includegraphics[width=\linewidth]{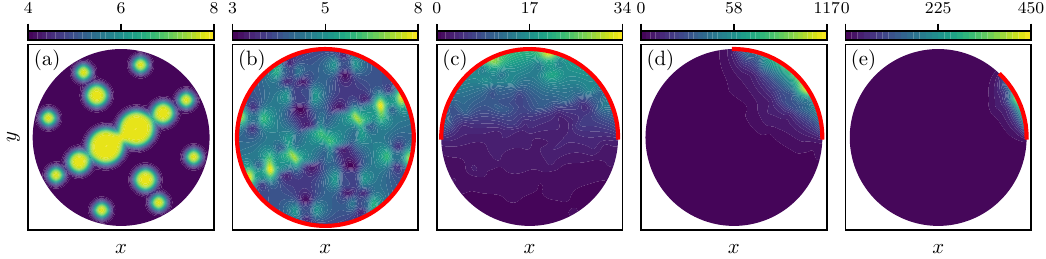}
    \caption{(a) True conductivity phantom consisting of several smoothed inclusions of varying sizes.
(b)–(e) Corresponding noise-free electrical energy density measurements $h_{1,1}$ for full, half, quarter, and eighth boundary-view configurations, respectively. In each case, the red curve indicates the location and extent of the applied boundary input used to generate the measurements.}
    \label{fig:signals}
\end{figure}

\begin{figure*}[t]
\centering
\graphicspath{{Figures/reconstructions/}}

\setlength{\tabcolsep}{2pt}
\renewcommand{\arraystretch}{0}

\newcommand{\reconw}{0.24\textwidth}

\begin{tabular}{cc}
Boundary functions $u_1\vert_{\partial \Omega}$& 
Boundary functions $u_2\vert_{\partial \Omega}$  \\[6pt]
\subcaptionbox{\label{fig:boundary_input_smooth}}{%
\includegraphics[trim={0 4cm 0 5cm},clip,width=0.49\linewidth]{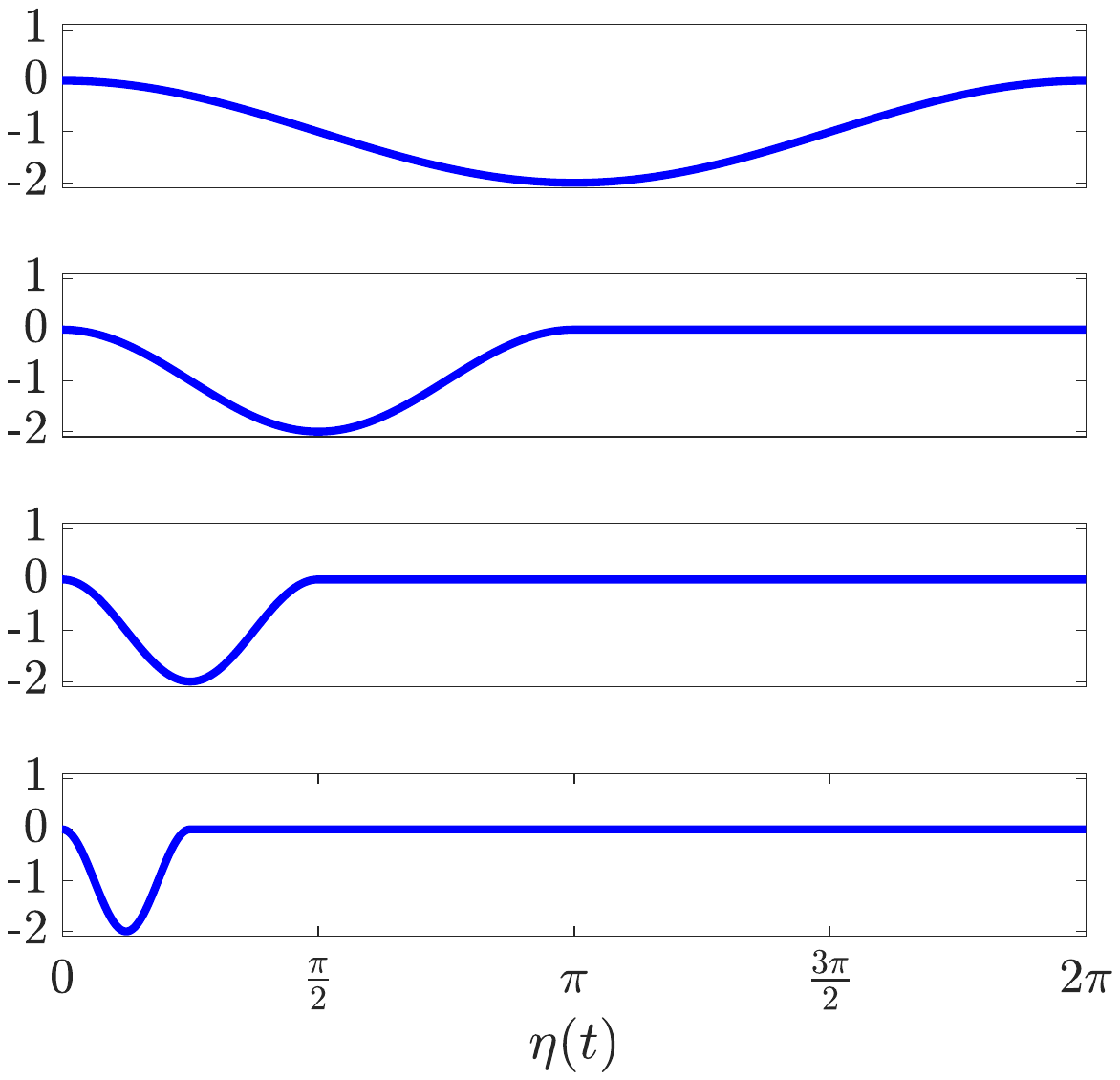}} &
\subcaptionbox{\label{fig:boundary_input_sharp}}{%
\includegraphics[trim={0 4cm 0 5cm},clip,width=0.49\linewidth]{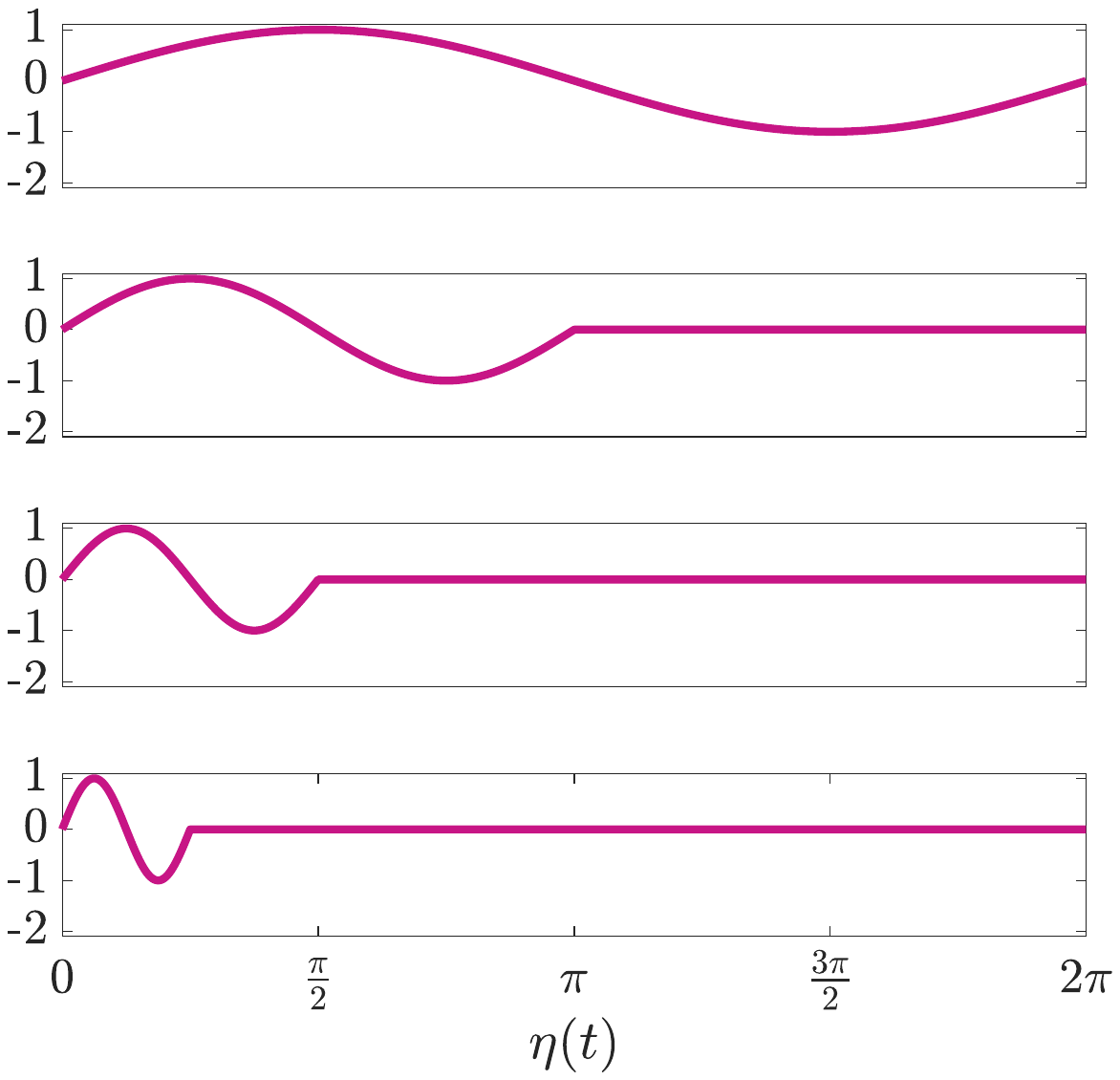}}

\end{tabular}

\caption{The boundary functions $u_1\vert_{\partial \Omega}$ and $u_2\vert_{\partial \Omega}$ used for generating the power density data for the numerical examples. Each row corresponds to a different limited view setting.}
\label{fig:boundaryfuns}
\end{figure*}

To generate noisy measurements, we follow the likelihood formulations described in \Cref{sec:likelihood-posterior}. We consider two types of noise levels: one associated with the $L^1$ likelihood model and another associated with the $L^2$ likelihood model. In both cases, the noise realization is constructed by first drawing an i.i.d. sample $\boldsymbol{e} \sim \mathcal{N}(0, I_{N_{\text{FEM}}})$ (cf. \eqref{eq:AET-discrete}).
We then generate the noisy data according to
\begin{equation}\label{eq:noisemodel} \boldsymbol{y}_{1,1}^{\text{obs}} = h_{1,1} + \tau_{1,1}^{\text{noise}} \varepsilon_{1,1}, \end{equation}
where $\tau_{1,1}$ is defined in \eqref{eq:noise-level} for $\ell = 1$ or $2$, $\varepsilon_{1,1}$ is the FEM noise function obtained by assembling basis functions with expansion coefficients $\boldsymbol{\varepsilon}_{1,1}$ as in \eqref{eq:noise_FEM}. 
In this work, we consider a $1\%$ relative noise level ($d_{\text{noise}} = 0.01$) for simulations using the $L^2$ likelihood, and a $0.1\%$ noise level ($d_{\text{noise}}=0.001$) for simulations using the $L^1$ likelihood (cf. \Cref{eq:noise-level}). Although these noise levels are not directly comparable across the two models, they are chosen to produce a similar visual level of perturbation in the resulting noisy signals.

To construct the prior distribution, we first compute the basis functions of the covariance operator $\mathcal C$ in \eqref{eq:matern-cov} used in the KL expansion. This is done numerically by solving the eigenvalue problem $(\tau I - \Delta) e_i = \mu_i e_i$ with homogeneous Dirichlet boundary conditions. This choice of boundary conditions enforces the conductivity to match the background value at the domain boundary, thereby ensuring that inclusions remain inside the domain. The eigenvalue problem is solved using the built-in eigensolver in FEniCS \cite{alnaes2015fenics,logg2012automated}, which is based on Krylov subspace iterative methods. The inverse relation \eqref{eq:eig_inverse_relation} is then used to obtain the KL expansion coefficients. The resulting eigenvalues are normalized so that $\sum_i \lambda_i = 1$. In all FEM simulations, the KL expansion is truncated after 300 terms. This choice ensures that at least 95\% of the variance of $X$ is retained under the truncated approximation. 

We consider 2 types of priors, a log-Gaussian prior, i.e., $F_1(X) =3+ 2\exp(X)$, which is as a linear transformed variant to \eqref{eq:pushforward-exp}, and piecewise-constant and near piecewise-constant priors, with $F_2$ and $F_3$ as described in \Cref{eq:prior-piecewise-constant}, with $\sigma^-=4$ and $\sigma^+=8$. Following the likelihood construction in \Cref{sec:likelihood-posterior}, we can formulate the posterior distribution, with respect to KL-expansion coefficients as
\begin{equation}
        \pi_{(X|Y_{1,1}=\boldsymbol{y}_{1,1})}(\boldsymbol{x}) \propto \exp\left( -  \frac{\| \mathcal G_{1,1}[ F_p(\boldsymbol{x}) ] - \boldsymbol{y}_{1,1} \|_{L^\ell}^\ell}{\ell (\tau_{i,j}^{\text{noise}})^\ell} - \frac{1}{2}\|\boldsymbol{x}\|_2^2 \right), \qquad p=1,2 \text{ or } 3.
\end{equation}
where $\boldsymbol{x}$ is the vector collecting all KL-expansion coefficients.

To explore the posterior distribution we use the pCN sampler described in \Cref{sec:mcmc}. A crucial component of an effective sampling strategy is the choice of the step size $\beta$. We employ an adaptive step-size adjustment scheme following \cite{sherlock2009optimal}. The sampling procedure is divided into a warm-up phase, during which the step size is adaptively tuned to achieve a target acceptance rate of 23\%. After this phase, the step size is fixed, and sampling proceeds in the online phase to collect the desired number of posterior samples. 

In all experiment constructed with an $L^2$-likelihood in this paper, we use 1K warm-up iterations followed by 5K sampling iterations in the online phase while for those experiments with an $L^1$-likelihood we use same number of warm-up iterations but followed by 20K sampling iterations. The additional samples are due to the added complexity of the $L^1$-likelihood on the posterior. We observe that the step size typically stabilizes in the range $[0.09,\,0.3]$, indicating efficient exploration of the posterior distribution with the pCN method. This behavior contrasts with standard EIT inverse problems under similar priors, where the pCN step size often becomes very small, reflecting a more complex posterior structure. This observation suggests that the AET formulation leads to a posterior that is easier to explore numerically.

\begin{figure}
    \centering
    \includegraphics[width=\linewidth]{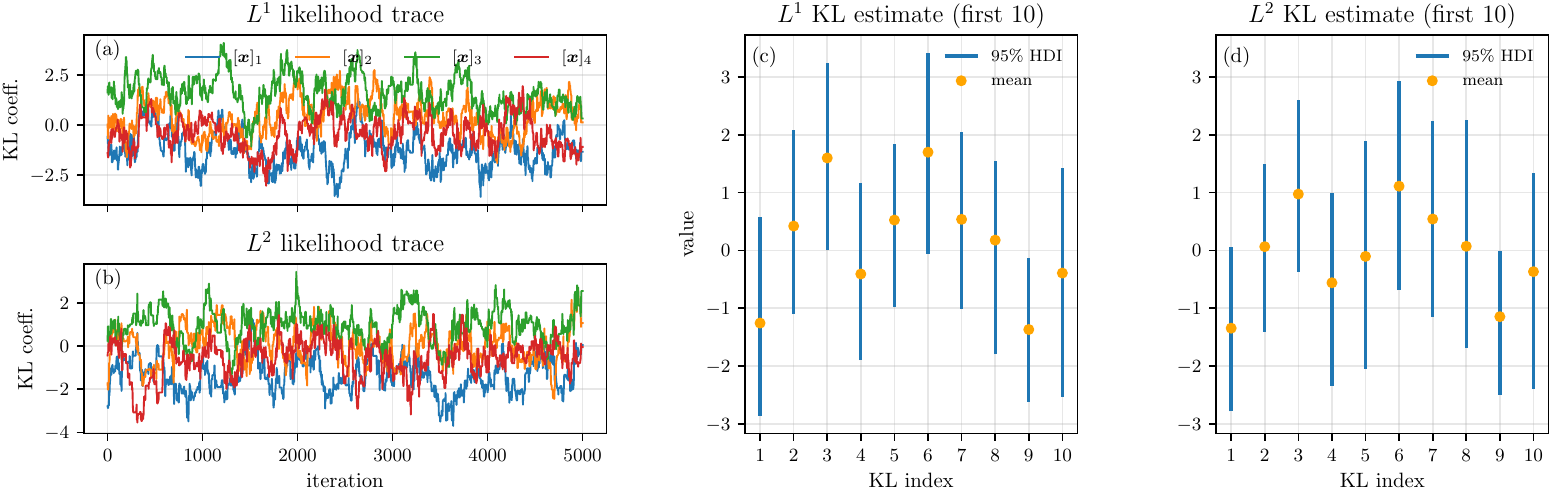}
    \caption{Diagnostics and sampling efficiency for the statistical AET inverse problem with a smooth prior. Panels (a) and (b) show trace plots for the first four components of $\boldsymbol{x}$ under the $L^1$ and $L^2$ likelihood constructions, respectively. Panels (c) and (d) display the posterior mean estimates of the first 10 KL coefficients of $\boldsymbol{x}$ together with the corresponding $95\%$ highest posterior density intervals (HDIs), illustrating the uncertainty in the coefficient estimates.}
    \label{fig:diagnostics_smooth}
\end{figure}

\begin{figure}
    \centering
    \includegraphics[width=\linewidth]{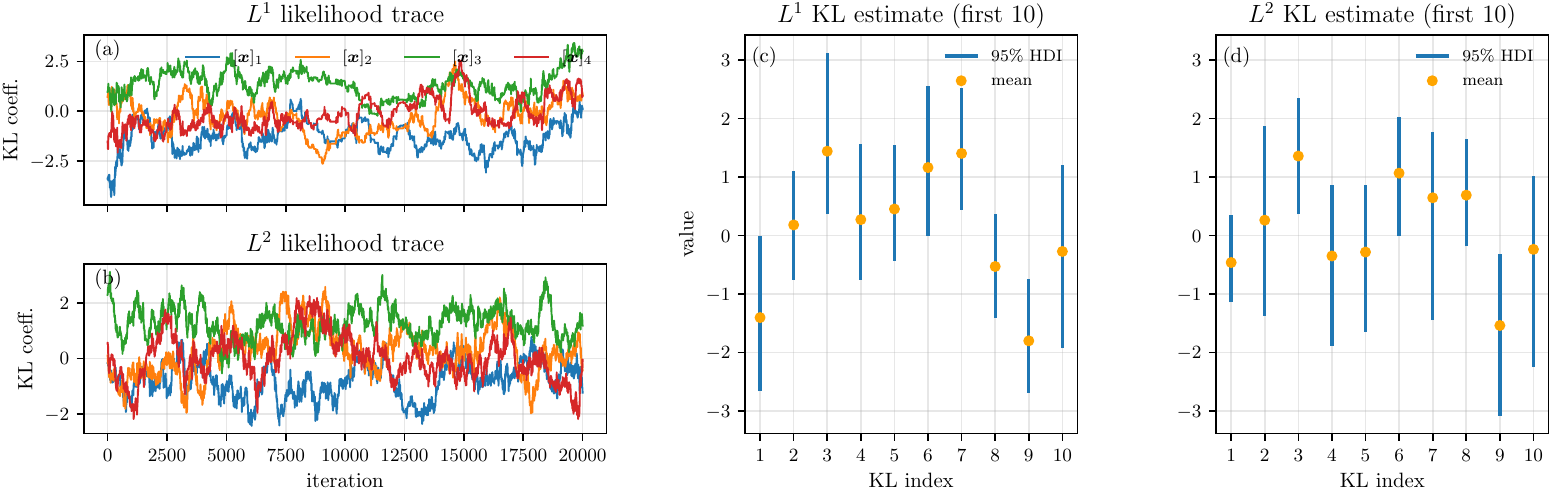}
    \caption{Diagnostics and sampling efficiency for the statistical AET inverse problem with a piecewise constant prior. Panels (a) and (b) show trace plots for the first four components of $\boldsymbol{x}$ under the $L^1$ and $L^2$ likelihood constructions, respectively. Panels (c) and (d) display the posterior mean estimates of the first 10 KL coefficients of $\boldsymbol{x}$ together with the corresponding $95\%$ highest posterior density intervals (HDIs), illustrating the uncertainty in the coefficient estimates.}
    \label{fig:diagnostics_piecewise_constant}
\end{figure}

Sample diagnostics for the $90^\circ$ limited-view configuration are shown for different priors and likelihood constructions in \Cref{fig:diagnostics_smooth} for the smooth prior and \Cref{fig:diagnostics_piecewise_constant} for the piecewise constant prior. Panels (a) and (b) display trace plots of the first 4 KL expansion coefficients. The smooth prior exhibits improved mixing behavior compared with the piecewise constant prior, and similarly the $L^2$ likelihood shows better mixing than the $L^1$ likelihood. The most challenging sampling scenario arises for the combination of the $L^1$ likelihood and the piecewise constant prior, where longer correlations between samples are observed. Panels (c) and (d) show the posterior mean estimates of the KL coefficients together with the corresponding highest posterior density (HPD) intervals, providing a measure of the associated uncertainty. The KL coefficient estimates are broadly consistent across the different configurations, indicating that the proposed statistical framework provides stable AET reconstructions. We also observe slightly narrower HPD intervals for the $L^1$ likelihood, which may reflect poorer mixing of the samples and a resulting underestimation of uncertainty. Overall, these diagnostics suggest that the sampling strategy employed in this work effectively explores the posterior distribution. We report that the diagnostics for other cases report in this paper are similar.

Figure \ref{fig:reconstructions-grid} shows the posterior mean reconstructions for the $L^1$ and $L^2$ likelihoods for both the smooth prior and the piecewise-constant prior. The corresponding posterior standard deviation fields are illustrated in Figure \ref{fig:reconstructions-std-grid}. In both figures, the red boundary curve indicates the boundary of control, $\Gamma_1$, along which the boundary function $f_1$ is applied. From the posterior mean reconstructions it is visible that, as the size of $\Gamma_1$ decreases, fewer inclusions of the true conductivity phantom are reconstructed. In particular, for the full boundary of control, $\Gamma_1=\partial \Omega$, more of the interior inclusions are reconstructed; however, their shapes do not match the original circular shapes of the inclusions. This is also reflected in the posterior standard deviation fields, in particular when using the piecewise-constant prior (see Figure \ref{fig:std-full-L1-pwc} and Figure \ref{fig:std-full-L2-pwc}), where there is uncertainty around the exact shape and location of the four central inclusions.

For the limited view settings, the features close to the boundary of control are reconstructed very well, while regions further away are reconstructed poorly. The fact that features close to $\Gamma_1$ are reconstructed better in the limited view settings than in the full view setting can be explained by the choice of the boundary functions $f_1$. For each limited view setting, $f_1$ is chosen in accordance with Lemma \ref{lem:CriticPoints} such that it extends continuously to zero at the endpoints of $\Gamma_1$ and can be split into a non-decreasing and a non-increasing part along $\Gamma_1$. For the full view setting, the boundary function is given by $f_1=\cos(\eta(t))-1$. For the limited view settings, higher-frequency variations of this function are used, namely $f_1=\cos(\ell_{\text{bound}}\, \eta(t))-1$ with $\ell_{\text{bound}}=2,4,8$, respectively. The use of higher-frequency boundary functions illuminates the region more effectively and thus yields more information for the reconstruction.

Furthermore, we observe that the piecewise-constant prior performs better in reconstructing distinct inclusions, whereas the smooth prior smooths out the inclusions to a large extent. The posterior standard deviation fields corresponding to the piecewise-constant prior are also more informative: large blue regions indicate areas where the reconstruction is trustworthy, while for some reconstructed inclusions there remains uncertainty regarding the exact boundary curve and shape. In the smooth case, less information is visible, as the posterior standard deviation fields are smoothed out to a similar extent as the reconstructions.

The $L^1$ and $L^2$ likelihood functions perform similarly overall. However, in the limited view settings, the reconstructions obtained with the $L^1$ likelihood contain more inclusions of the true conductivity phantom. For example, for the smallest boundary of control (Figure \ref{fig:recon-eighth-L1-pwc}), the inclusion close to $\Gamma_1$ is reconstructed very well, and the inclusion slightly further away is also visible. In contrast, for the $L^2$ likelihood (Figure \ref{fig:recon-eighth-L2-pwc}), only the inclusion closest to $\Gamma_1$ is reconstructed. The high certainty of these inclusions is also reflected in the corresponding posterior standard deviation fields (Figure \ref{fig:std-eighth-L1-pwc} and Figure \ref{fig:std-eighth-L2-pwc}). Nevertheless, there remains some uncertainty regarding the precise location of the inclusion boundaries and the shape of the inclusion further away.

Despite the Bayesian method performing slightly better in the full view setting with the $L^2$ likelihood function and the piecewise-constant prior, we select the $L^1$ likelihood function with the piecewise-constant prior as the preferred method due to its superior performance in the limited view setting. 
In the limited view case, the $L^1$ likelihood (Figure \ref{fig:recon-quarter-L1-pwc} and Figure \ref{fig:recon-eighth-L1-pwc}) outperforms the $L^2$ likelihood (Figure \ref{fig:recon-quarter-L2-pwc} and Figure \ref{fig:recon-eighth-L2-pwc}), as a larger portion of the conductivity and more small-scale features near the boundary are recovered. 
Since our primary objective is robust performance in the limited view setting, while also maximizing the regions in the posterior standard deviation where the reconstruction can be trusted—as achieved with the $L^1$ likelihood—we adopt the $L^1$ likelihood function together with the piecewise-constant prior in the following.

\begin{figure*}[t]
\centering
\graphicspath{{Figures/reconstructions/}}

\setlength{\tabcolsep}{2pt}
\renewcommand{\arraystretch}{0}

\newcommand{\reconw}{0.24\textwidth}

\begin{tabular}{cccc}

$L^1$ smooth& 
$L^1$ piecewise &
$L^2$ smooth& 
$L^2$ piecewise  \\[6pt]

\subcaptionbox{\label{fig:recon-full-L1-smooth}}{%
\includegraphics[width=\reconw]{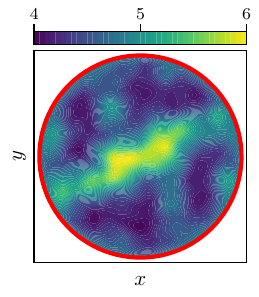}} &
\subcaptionbox{\label{fig:recon-full-L1-pwc}}{%
\includegraphics[width=\reconw]{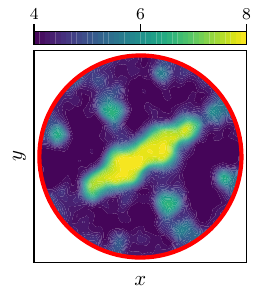}} &
\subcaptionbox{\label{fig:recon-full-L2-smooth}}{%
\includegraphics[width=\reconw]{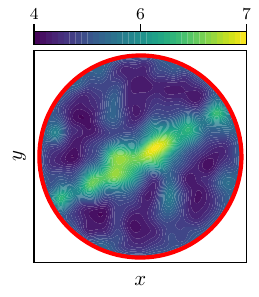}} &
\subcaptionbox{\label{fig:recon-full-L2-pwc}}{%
\includegraphics[width=\reconw]{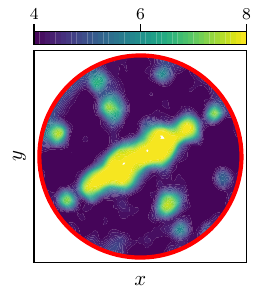}} \\[2pt]

\subcaptionbox{\label{fig:recon-half-L1-smooth}}{%
\includegraphics[width=\reconw]{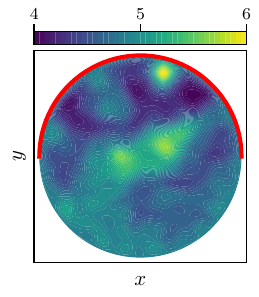}} &
\subcaptionbox{\label{fig:recon-half-L1-pwc}}{%
\includegraphics[width=\reconw]{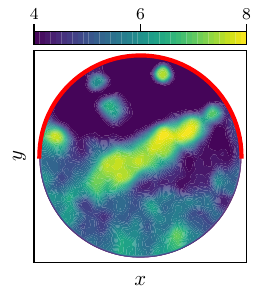}} &
\subcaptionbox{\label{fig:recon-half-L2-smooth}}{%
\includegraphics[width=\reconw]{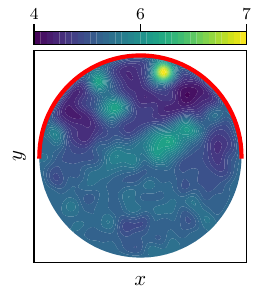}} &
\subcaptionbox{\label{fig:recon-half-L2-pwc}}{%
\includegraphics[width=\reconw]{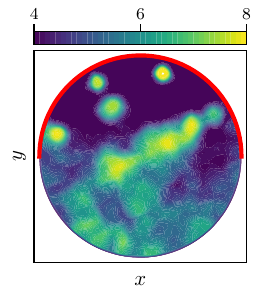}}\\[2pt]

\subcaptionbox{\label{fig:recon-quarter-L1-smooth}}{%
\includegraphics[width=\reconw]{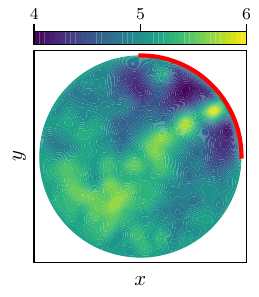}} &
\subcaptionbox{\label{fig:recon-quarter-L1-pwc}}{%
\includegraphics[width=\reconw]{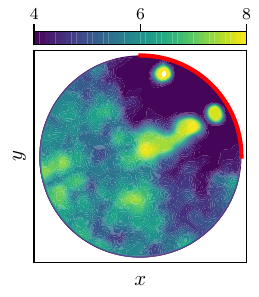}} &
\subcaptionbox{\label{fig:recon-quarter-L2-smooth}}{%
\includegraphics[width=\reconw]{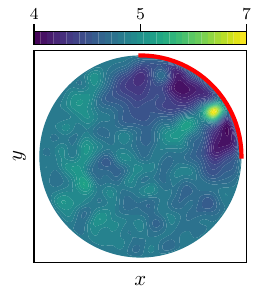}} &
\subcaptionbox{\label{fig:recon-quarter-L2-pwc}}{%
\includegraphics[width=\reconw]{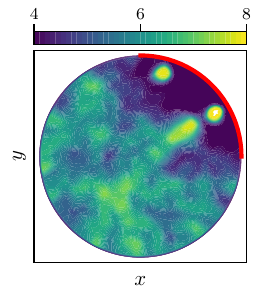}} \\[2pt]

\subcaptionbox{\label{fig:recon-eighth-L1-smooth}}{%
\includegraphics[width=\reconw]{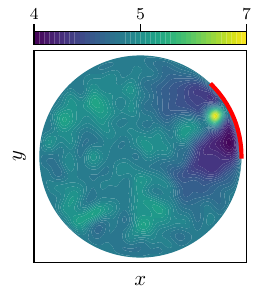}} &
\subcaptionbox{\label{fig:recon-eighth-L1-pwc}}{%
\includegraphics[width=\reconw]{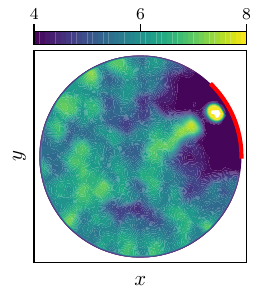}} &
\subcaptionbox{\label{fig:recon-eighth-L2-smooth}}{%
\includegraphics[width=\reconw]{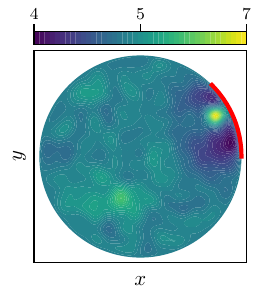}} &
\subcaptionbox{\label{fig:recon-eighth-L2-pwc}}{%
\includegraphics[width=\reconw]{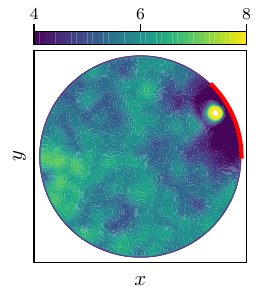}} 

\end{tabular}

\caption{Posterior mean reconstructions for different likelihood models and limited-view configurations.
The first two columns correspond to reconstructions obtained with the $L^1$ likelihood (smooth and piecewise-constant priors, respectively), while the next two columns correspond to the $L^2$ likelihood with the same prior choices. The rows represent different limited-view settings, from full to increasingly restricted boundary coverage. In each panel, the red boundary curve indicates the extent of the applied boundary input.}
\label{fig:reconstructions-grid}
\end{figure*}

\begin{figure*}[t]
\centering
\graphicspath{{Figures/reconstructions/}}

\setlength{\tabcolsep}{2pt}
\renewcommand{\arraystretch}{0}

\newcommand{\reconw}{0.24\textwidth}

\begin{tabular}{cccc}

$L^1$ smooth& 
$L^1$ piecewise &
$L^2$ smooth& 
$L^2$ piecewise  \\[6pt]

\subcaptionbox{\label{fig:std-full-L1-smooth}}{%
\includegraphics[width=\reconw]{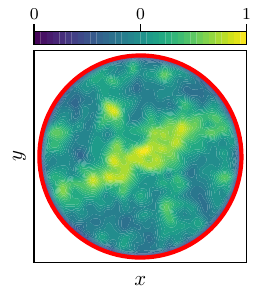}} &
\subcaptionbox{\label{fig:std-full-L1-pwc}}{%
\includegraphics[width=\reconw]{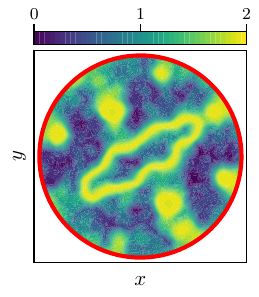}} &
\subcaptionbox{\label{fig:std-full-L2-smooth}}{%
\includegraphics[width=\reconw]{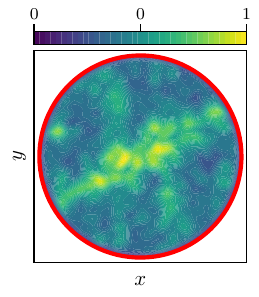}} &
\subcaptionbox{\label{fig:std-full-L2-pwc}}{%
\includegraphics[width=\reconw]{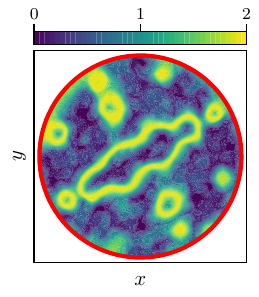}} \\[2pt]

\subcaptionbox{\label{fig:std-half-L1-smooth}}{%
\includegraphics[width=\reconw]{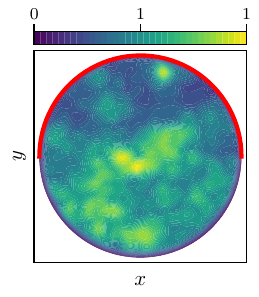}} &
\subcaptionbox{\label{fig:std-half-L1-pwc}}{%
\includegraphics[width=\reconw]{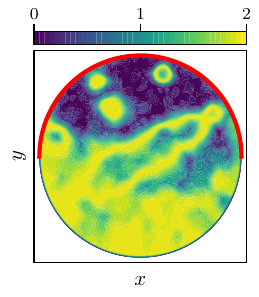}} &
\subcaptionbox{\label{fig:std-half-L2-smooth}}{%
\includegraphics[width=\reconw]{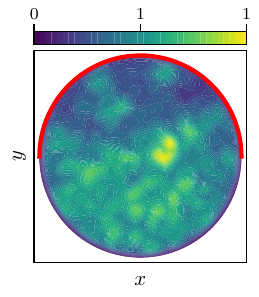}} &
\subcaptionbox{\label{fig:std-half-L2-pwc}}{%
\includegraphics[width=\reconw]{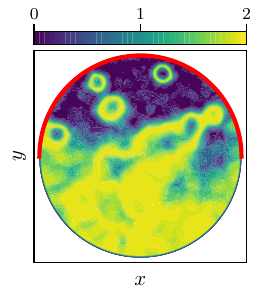}} \\[2pt]

\subcaptionbox{\label{fig:std-quarter-L1-smooth}}{%
\includegraphics[width=\reconw]{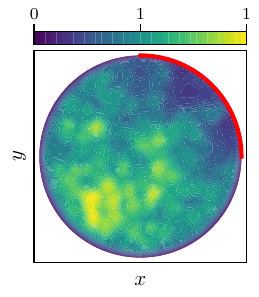}} &
\subcaptionbox{\label{fig:std-quarter-L1-pwc}}{%
\includegraphics[width=\reconw]{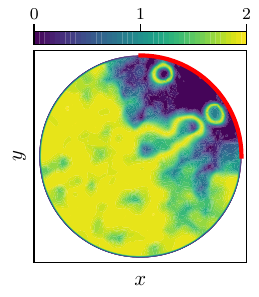}} &
\subcaptionbox{\label{fig:std-quarter-L2-smooth}}{%
\includegraphics[width=\reconw]{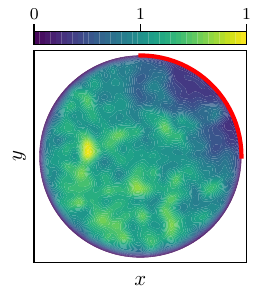}} &
\subcaptionbox{\label{fig:std-quarter-L2-pwc}}{%
\includegraphics[width=\reconw]{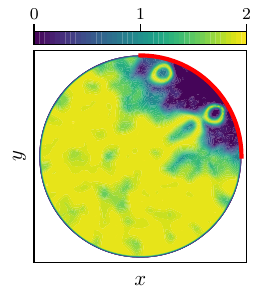}} \\[2pt]

\subcaptionbox{\label{fig:std-eighth-L1-smooth}}{%
\includegraphics[width=\reconw]{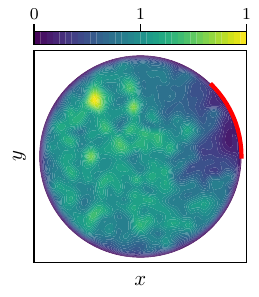}} &
\subcaptionbox{\label{fig:std-eighth-L1-pwc}}{%
\includegraphics[width=\reconw]{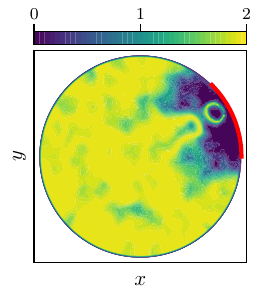}} &
\subcaptionbox{\label{fig:std-eighth-L2-smooth}}{%
\includegraphics[width=\reconw]{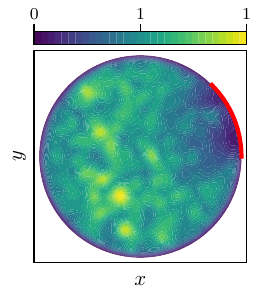}} &
\subcaptionbox{\label{fig:std-eighth-L2-pwc}}{%
\includegraphics[width=\reconw]{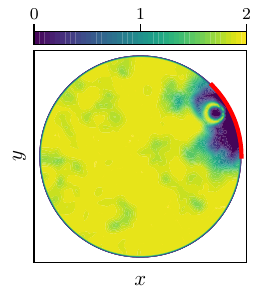}} \\[2pt]

\end{tabular}

\caption{Posterior standard deviation fields corresponding to the reconstructions in \Cref{fig:reconstructions-grid}.
The first two columns correspond to the $L^1$ likelihood (smooth and piecewise-constant priors), the next two columns to the $L^2$ likelihood with the same priors. Rows represent decreasing view angles, with the red boundary curve indicating the extent of the applied boundary input.}
\label{fig:reconstructions-std-grid}
\end{figure*}

\subsection{Comparison of Bayesian and deterministic reconstructions}
In the following, we compare the deterministic reconstruction procedure from Section \ref{sec:deterministic} with the Bayesian method using an $L^1$ likelihood and a piecewise-constant prior. The Lipschitz stability result in \cite[Thm 3.2]{Bal2013} suggests that the analytical reconstruction in the full-view setting is robust to noise, although this is established in a stronger norm than used in our data misfit. However, the deterministic method does not explicitly account for noise, and its performance is expected to degrade, particularly in the limited-view setting. In contrast, the Bayesian approach incorporates noise through the likelihood and prior, and is therefore expected to yield more robust reconstructions.

We consider a similar discretization for the analytical reconstruction procedure as for the Bayesian reconstruction procedure in order to ensure comparability of the results. The main difference lies in the mesh discretization: for the deterministic method, we use an unstructured triangulated mesh generated from the geometric description of the circle using \texttt{mshr}, a module of the FEniCS Project \cite{alnaes2015fenics}. This mesh is comparable in size to the one used for the Bayesian method, having $7{,}253$ degrees of freedom. Apart from the mesh discretization, the implementations remain the same.

For the deterministic reconstruction procedure, two boundary functions $f_1$ and $f_2$ are required. Therefore, we complement the boundary function $f_1$ shown in Figure \ref{fig:boundary_input_smooth} with the corresponding boundary function $f_2$ shown in Figure \ref{fig:boundary_input_sharp} to obtain our measurements $h_{i,j}$, $i,j=1,2$. Since the $L^1$ likelihood function was selected in the previous subsection as providing the best limited view reconstructions for a relative noise level of $0.1\%$, we consider the noise model in \eqref{eq:noisemodel} with $d_{\text{noise}}=0.001$ and $\ell=1$ for the deterministic method to obtain our noisy measurement matrix $\boldsymbol{Y}^{\mathrm{obs}}$. Here $\boldsymbol{Y}^{\mathrm{obs}}$ is a matrix valued function defined such that each entry corresponds to the FEM function $\boldsymbol{y}^{\mathrm{obs}}_{i,j}$: $[\boldsymbol{Y}^{\mathrm{obs}}]_{i,j}:=\boldsymbol{y}^{\mathrm{obs}}_{i,j}$, for $i,j=1,2$. Invertibility of the measurement matrix $\boldsymbol{Y}^{\mathrm{obs}}$ then refers to point-wise invertibility, i.e. $\det (\boldsymbol{Y}^{\mathrm{obs}}(\xi))\neq 0$ for all $\xi\in \Omega$.

The boundary functions $f_1$ and $f_2$ are chosen in accordance with Lemma \ref{lem:Jacobian} to ensure that the Jacobian condition $\det[\nabla u_1 \, \nabla u_2]>0$ is satisfied to guarantee invertibility the noise free matrix $\boldsymbol{H}$. The perturbed measurement matrix $\boldsymbol{Y}^{\mathrm{obs}}$ does no longer satisfy this condition and since the deterministic method requires inversions of $\boldsymbol{Y}^{\mathrm{obs}}$, we need to modify our measurements. We enforce positive definiteness by setting eigenvalues of $\boldsymbol{Y}^{\mathrm{obs}}$ below the threshold $b=0.002$ to zero to obtain the measurement matrix $\boldsymbol{\widetilde{Y}}^{\mathrm{obs}}$. The reconstructions of $\sigma$ obtained from $\boldsymbol{\widetilde{Y}}^{\mathrm{obs}}$ following the analytical reconstruction procedure are illustrated in the left column of Figure \ref{fig:reconstructions-comparison}. 
These are compared to the Bayesian reconstructions using the $L^1$ likelihood function with the piecewise-constant prior in the middle column of the same figure with the corresponding posterior standard deviation fields shown in the right column. 

The deterministic reconstruction method gives a really good reconstruction in the full view setting (Figure \ref{fig:L1_analytic_fullFinal}), while the Bayesian reconstruction (Figure \ref{fig:L1_piecewise_constnat_full_mean}) is not performing onpar with the deterministic case and only recovers the main inclusions in the center and a few inclusions towards the boundary; moreover, the shapes of the inclusions are not captured correctly. A similar trend is visible for the half view setting, where the boundary of control $\Gamma_1$ is half of the boundary. Here the deterministic method in Figure \ref{fig:L1_analytic_halfFinal} recovers most inclusions, but the values of the inclusions are only correct close towards $\Gamma_1$, otherwise they are too low and towards $\partial \Omega \setminus \Gamma_1$ artifacts appear from the noise. For the Bayesian reconstruction in Figure \ref{fig:L1_piecewise_constnat_half_mean} all inclusions in the upper half are recovered, but their shapes are not correct and most of the values of the inclusions are too low.\par

For the limited view settings, where $\Gamma_1$ is either a quarter of the boundary (the third row of Figure \ref{fig:reconstructions-comparison}) or an eighth of the boundary (the fourth row of Figure \ref{fig:reconstructions-comparison}) the Bayesian reconstructions are better than the deterministic reconstructions. The deterministic reconstructions in Figure \ref{fig:L1_analytic_quarterFinal} and \ref{fig:L1_analytic_eighthFinal} are dominated by noise so that only one inclusion towards $\Gamma_1$ can be recovered (when $\Gamma_1$ is a quarter of the boundary) or no inclusions can be recovered (when $\Gamma_1$ is one eighth of the boundary). In contrast, the Bayesian reconstructions in Figure \ref{fig:L1_piecewise_constnat_quarter_mean} and \ref{fig:L1_piecewise_constnat_eighth_mean} recover the region closest towards $\Gamma_1$ very well.
When $\Gamma_1$ covers one eighth of the boundary (Figure \ref{fig:L1_piecewise_constnat_eighth_mean}), a similar behavior is observed, but the region where the conductivity is well recovered is smaller. The background value is again reconstructed correctly, and the inclusion closest to $\Gamma_1$ is well recovered. Another inclusion toward the center is detected, but its shape is not accurately resolved. As in the previous case, the reconstruction transitions into a smooth and blurred field away from $\Gamma_1$, where identifying inclusions becomes difficult.

This behavior is reflected in the posterior standard deviation field in Figure \ref{fig:L1_piecewise_constnat_quarter_std}. In regions close to $\Gamma_1$, there is high certainty in the reconstructed background values and in the interior of the inclusions, with uncertainty mainly concentrated along the boundaries of the inclusions. In contrast, in the regions where the reconstruction appears smooth and blurred, the posterior standard deviation is significantly higher, indicating that these parts of the reconstruction cannot be reliably trusted. A similar trend is observed when $\Gamma_1$ covers one eighth of the boundary, where the region of low uncertainty is reduced and higher uncertainty appears closer to $\Gamma_1$.

In summary, these results of the Bayesian method illustrate that the extent of the domain over which the reconstruction can be trusted depends strongly on the portion of the boundary where measurements are available, and that the posterior uncertainty provides a clear, quantitative indicator of which parts of the reconstruction are reliable. To the best of our knowledge, this is the first time in AET reconstructions that one can directly assess how much of the domain can be trusted from the posterior uncertainty associated with the reconstruction.

Overall, the deterministic method performs better than the Bayesian method for the full and half view setting, while the Bayesian method is superior for the narrower limited view settings with quantified uncertainty for indicating regions of trust. The Bayesian method yields better reconstructions in the narrower limited view settings than for larger boundary of control. This is due to the fact that the boundary functions $f_1=\cos(\ell_{\text{bound}}\, \eta(t))-1$ with $\ell_{\text{bound}}=2,4,8$ are chosen in accordance with Lemma \ref{lem:CriticPoints} so that exactly one pulse is created that divides $f_1$ into one non-increasing and one non-decreasing part along $\partial \Omega$. For the narrower limited view settings this implies that $f_1$ is more oscillating and thus has a tendency to illuminate $\Omega$ better and result in more informative measurements $h_{1,1}$. However, our goal was to choose the boundary functions so that the analytical reconstruction procedure is feasible and so that there are no points $\xi_0\in \Omega$, where s $h_{1,1}$ does not contain any information, because $\nabla u(\xi_0)=0$.\par

The better performance of the deterministic method for large $\Gamma_1$ is also could be explained by the fact that it uses measurements corresponding to two boundary functions imposed rather than only one boundary function imposed for the Bayesian method. However, this claim needs further investigation and our intention was to demonstrate what is possible with the Bayesian method when using such a low amount of measurements.

\begin{figure*}[t]
\centering
\graphicspath{{Figures/reconstructions/}}

\setlength{\tabcolsep}{2pt}
\renewcommand{\arraystretch}{0}

\newcommand{\reconw}{0.24\textwidth}

\begin{tabular}{ccc}

Analytic reconstruction& 
Posterior Mean &
Posterior Std \\[6pt]


\subcaptionbox{\label{fig:L1_analytic_fullFinal}}{%
\includegraphics[width=\reconw]{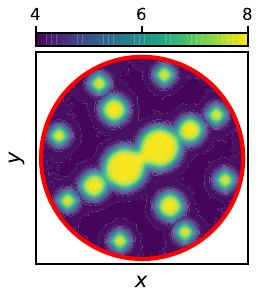}} &
\subcaptionbox{\label{fig:L1_piecewise_constnat_full_mean}}{%
\includegraphics[width=\reconw]{L1_piecewise_constnat_full_mean.pdf}} &
\subcaptionbox{\label{fig:L1_piecewise_constnat_full_std}}{%
\includegraphics[width=\reconw]{L1_piecewise_constnat_full_std.pdf}} \\[2pt]

\subcaptionbox{\label{fig:L1_analytic_halfFinal}}{%
\includegraphics[width=\reconw]{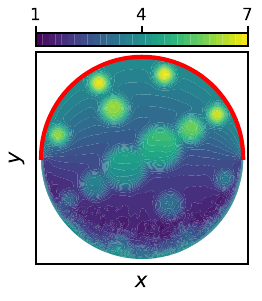}} &
\subcaptionbox{\label{fig:L1_piecewise_constnat_half_mean}}{%
\includegraphics[width=\reconw]{L1_piecewise_constnat_half_mean.pdf}} &
\subcaptionbox{\label{fig:L1_piecewise_constnat_half_std}}{%
\includegraphics[width=\reconw]{L1_piecewise_constnat_half_std.pdf}} \\[2pt]

\subcaptionbox{\label{fig:L1_analytic_quarterFinal}}{%
\includegraphics[width=\reconw]{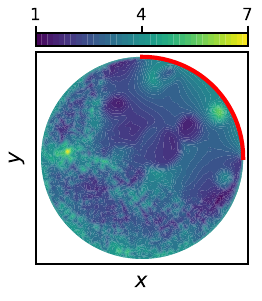}} &
\subcaptionbox{\label{fig:L1_piecewise_constnat_quarter_mean}}{%
\includegraphics[width=\reconw]{L1_piecewise_constnat_quarter_mean.pdf}} &
\subcaptionbox{\label{fig:L1_piecewise_constnat_quarter_std}}{%
\includegraphics[width=\reconw]{L1_piecewise_constnat_quarter_std.pdf}} \\[2pt]

\subcaptionbox{\label{fig:L1_analytic_eighthFinal}}{%
\includegraphics[width=\reconw]{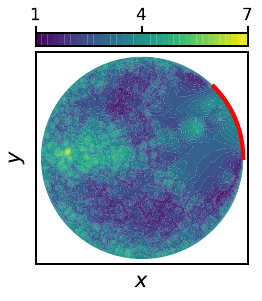}} &
\subcaptionbox{\label{fig:L1_piecewise_constnat_eighth_mean}}{%
\includegraphics[width=\reconw]{L1_piecewise_constnat_eighth_mean.pdf}} &
\subcaptionbox{\label{fig:L1_piecewise_constnat_eighth_std}}{%
\includegraphics[width=\reconw]{L1_piecewise_constnat_eighth_std.pdf}} \\[2pt]

\end{tabular}

\caption{Comparison between analytic reconstruction (first column) and posterior mean reconstruction corresponding to $L^1$ likelihood with a piecewise-constant prior (second column). The third column shows the corresponding posterior standard deviation field. The rows represent different limited-view settings, from full to increasingly restricted boundary coverage. In each panel, the red boundary curve indicates the extent of the applied boundary input.}
\label{fig:reconstructions-comparison}
\end{figure*}

\section{Conclusion}
\label{sec:conclusion}

In this work, we studied the acousto–electric tomography (AET) inverse problem within a Bayesian framework and investigated its reconstruction performance using both smooth and piecewise-constant priors across different limited-view settings. In particular, we considered likelihood constructions based on both $L^1$- and $L^2$-type data misfit norms and established Bayesian well-posedness for both formulations.

Through a series of numerical experiments, we analyzed the resulting posterior distributions and compared Bayesian reconstructions with deterministic analytical reconstruction methods. The experiments show that in the full and half view settings the deterministic reconstruction method performs better and recovers most inclusions accurately, whereas the Bayesian reconstructions capture only the main features and often underestimate the inclusion values or distort their shapes. This behavior is partly explained by the fact that the deterministic method uses measurements corresponding to two imposed boundary functions, while the Bayesian method in our experiments relies on data corresponding to only a single EIT measurement.

In contrast, for more restricted limited-view configurations the Bayesian approach performs more robustly. 
While the deterministic reconstructions become dominated by noise as the accessible boundary decreases, the Bayesian method is still able to recover the region of the domain closest to the accessible boundary $\Gamma_1$ and provides uncertainty information through the posterior standard deviation, allowing one to assess which parts of the reconstruction are reliable. In particular, the posterior uncertainty concentrates around inclusion boundaries and regions further away from $\Gamma_1$, reflecting the reduced information content of the measurements in these areas.

The numerical diagnostics further indicate that the posterior distributions arising in the AET formulation can be explored efficiently with the pCN sampling method, with stable step sizes and reasonable mixing behavior across different priors and likelihood constructions.

Overall, the results suggest that the deterministic reconstruction methods are preferable when information from two or more EIT measurements is available, while the Bayesian formulation provides a flexible framework in more challenging settings, particularly when measurements are limited and uncertainty quantification is required. 

\appendix
\section{The coupled step in AET} \label{sec:coupled}
This section is based on \cite[Sec. 9.2]{AlbertiCapdeboscq18}. The coupled step in AET aims at recovering the interior measurements $h_{i,j}(x)=\sigma(x) \nabla u_i(x) \cdot \nabla u_j(x)$ with $1\leq i,j \leq 2$ by combining EIT measurements at the boundary of $\partial \Omega$ and perturbing $\Omega$ with focused ultrasound waves. Each wave $p$ is generated at a source $S$ and satisfies the following boundary value problem
\begin{equation*}
            \begin{cases}
                (\partial_t^2-c^2(x)\Delta)p(x,t)=S(x,t), & (x,t)\in \R^2 \times [0,T],\\
                p(x,0)=\partial_t p(x,0) = 0, & x \in \R^2
            \end{cases},
\end{equation*}
where $c$ is the wave speed. Knowing the wave speed and the source function implies that the wave $p$ is known. As the ultrasound wave travels through $\Omega$ it compresses and expands the material. This induces a change in the conductivity that one refers to as the acousto-electric effect. This time dependent change in the conductivity $\sigma(x)$ is captured through the time dependent function $\sigma_p(x,t)$:
\begin{equation*}
    \sigma_p(x,t)=\sigma(x)(1+\eta p(x,t)),
\end{equation*}
where $\eta\geq 0$ is the acousto-electric coupling parameter. Recall that the electric potential $u_i(x)$ with $i=1,2$ associated to $\sigma(x)$ satisfies the boundary value problem
\begin{equation*}
\begin{cases}
    -\mathrm{div}(\sigma(x)\nabla u_i(x))=0 & \text{in }\Omega,\\
    u_i(x)=f_i(x) & \text{on }\partial \Omega.
    \end{cases}
\end{equation*}
We denote by $u_{i,p}$ with $i=1,2$ the potential associated to the time dependent conductivity $\sigma_p$ that solves the boundary value problem
\begin{equation*}
\begin{cases}
    -\mathrm{div}(\sigma_p(x,t)\nabla u_{i,p}(x,t))=0 & \text{in }\Omega\times [0,T],\\
    u_{i,p}(x,t)=f_i(x) & \text{on }\partial \Omega\times [0,T].
    \end{cases}
\end{equation*}
The EIT measurements correspond to measuring the currents $s_i(x)=\nu \cdot \sigma(x)\nabla u_i(x)\vert_{\partial \Omega}$ and $s_{i,p}(x,t)=\nu \cdot \sigma_p(x,t)\nabla u_{i,p}(x,t)\vert_{\partial \Omega}$ at the boundary while perturbing the domain with ultrasound waves. Here, $\nu$ denotes the unit outward normal to $\Omega$. For this purpose we investigate the cross-correlation of the boundary measurements. Using integration by parts, it can be shown that these satisfy the time series
\begin{equation}
    I(t)=\int_{\partial \Omega}\lp f_j  (\nu \cdot \sigma_p \nabla u_{i,p})-f_i  (\nu \cdot \sigma \nabla u_{j})\rp\,\mathrm{d}s=-\eta\int_{\Omega} p \sigma  \nabla u_i\cdot \nabla u_{j,p}\,\mathrm{d}x.\label{crossCorrBC}
\end{equation}
Assuming that the term $\eta p(x,t)$ is small, $\sigma_p$ is approximately $\sigma$, resulting in $u_p \approx u$. Then one obtains
\begin{equation*}
    I(t)\approx-\eta\int_{\Omega} p(x,t)\sigma(x)  \nabla u^i(x) \cdot \nabla u^j(x)\,\mathrm{d}x.
\end{equation*}

Placing the source at different locations, denoted by $S_m$, yields different acoustic fields $p_m$.
To be able to reconstruct the internal power densities
\begin{equation}
    h_{i,j}(x)=\sigma(x)\nabla u^i(x) \cdot \nabla u^j(x),\label{internalAET}
\end{equation}
one then needs to solve the following integral equations for $h_{i,j}$ 
\begin{equation*}
    I_m(t)=-\eta \int_{\Omega}p_m(x,t) h_{i,j}(x) \, \mathrm{d}x, \quad t \in [0,T].
\end{equation*}

\section*{Acknowledgments}
The authors thank Amal M. A. Alghamdi for assistance with implementing the AET problem in the CUQIpy framework, and Jakob S. Jørgensen for valuable discussions. Both authors are supported by the Research Council of Finland, B.~M.~Afkham under the grant number 371523 and H.~Schlüter under the Flagship of Advanced Mathematics for Sensing Imaging and Modelling grant 359208.

\clearpage
\newpage
\printbibliography


\end{document}